\documentclass[sn-mathphys,Numbered]{sn-jnl}% Math and Physical Sciences Reference Style
\usepackage{graphicx}%
\usepackage{multirow}%
\usepackage{amsfonts,amsmath,amssymb,amsthm}%
\usepackage{amsthm}%
\usepackage{mathrsfs}%
\usepackage[title]{appendix}%
\usepackage{xcolor}%
\usepackage{textcomp}%
\usepackage{manyfoot}%
\usepackage{booktabs}%
\usepackage{float}
\usepackage{listings}%
\usepackage{epsfig}
\usepackage{subcaption}
\usepackage{graphicx,colortbl,here}
\usepackage{soul}
\usepackage{mathtools}
\usepackage{amscd}
\usepackage{bbm}
\renewcommand{\theequation}{\arabic{section}.\arabic{equation}}

\newtheorem{theorem}{Theorem}[section]
\newtheorem{lemma}{Lemma}[section]
\newtheorem{corollary}{Corollary}[section]
\newtheorem{proposition}{Proposition}[section]
\newtheorem{remark}{Remark}[section]

\newtheorem{definition}{Definition}[section]

\begin{document}

\title[On the emergent dynamics of the infinite set of Kuramoto oscillators]{On the emergent dynamics of the infinite set of Kuramoto oscillators}

%%=============================================================%%
%% Prefix	-> \pfx{Dr}
%% GivenName	-> \fnm{Joergen W.}
%% Particle	-> \spfx{van der} -> surname prefix
%% FamilyName	-> \sur{Ploeg}
%% Suffix	-> \sfx{IV}
%% NatureName	-> \tanm{Poet Laureate} -> Title after name
%% Degrees	-> \dgr{MSc, PhD}
%% \author*[1,2]{\pfx{Dr} \fnm{Joergen W.} \spfx{van der} \sur{Ploeg} \sfx{IV} \tanm{Poet Laureate} 
%%                 \dgr{MSc, PhD}}\email{iauthor@gmail.com}
%%=============================================================%%

\author[1,2]{\fnm{Seung-Yeal} \sur{Ha}}\email{syha@snu.ac.kr}

\author*[1]{\fnm{Euntaek} \sur{Lee}}\email{tngkrqks21@snu.ac.kr}

\author[3]{\fnm{Woojoo} \sur{Shim}}\email{wjshim@knu.ac.kr}

\affil[1]{\orgdiv{Department of Mathematical Sciences}, \orgname{Seoul National University}, \orgaddress{ \city{Seoul}, \postcode{08826}, \country{Republic of Korea}}}

\affil[2]{\orgdiv{Research Institute of Mathematics}, \orgname{Seoul National University}, \orgaddress{ \city{Seoul}, \postcode{08826}, \country{Republic of Korea}}}

\affil[3]{\orgdiv{Department of Mathematics Education}, \orgname{Kyungpook National University}, \orgaddress{ \city{Daegu}, \postcode{41566}, \country{Republic of Korea}}}

%%==================================%%
%% sample for unstructured abstract %%
%%==================================%%

\abstract{We propose an infinite Kuramoto model for a countably infinite set of Kuramoto oscillators and study its emergent dynamics for two classes of network topologies. For a class of symmetric and row(or column)-summable network topology, we show that a homogeneous ensemble exhibits complete synchronization, and the infinite Kuramoto model can cast as a gradient flow, whereas we obtain a weak synchronization estimate, namely practical synchronization for a heterogeneous ensemble. Unlike with the finite Kuramoto model, phase diameter can be constant for some class of network topologies which is a novel feature of the infinite model. We also consider a second class of network topology (so-called a sender network) in which coupling strengths are proportional to a constant that depends only on sender's index number. For this network topology, we have a better control on emergent dynamics. For a homogeneous ensemble, there are only two possible asymptotic states, complete phase synchrony or bi-cluster configuration in any positive coupling strengths. In contrast, for a heterogeneous ensemble, complete synchronization occurs exponentially fast for a class of initial configuration confined in a quarter arc. }


%%\pacs[JEL Classification]{D8, H51}


\maketitle
\section{Introduction}\label{sec:1} 
\setcounter{equation}{0} 
Synchronization is one of collective behaviors in which weakly coupled oscillators adjust their rhythms via mutual interactions, and it is often observed in oscillatory systems such as the collection of fireflies, neurons and pacemaker cells, etc. However, despite its ubiquitous presence, its rigorous mathematical studies were begun in only a half century ago by two pioneers,  Arthur Winfree in 1967 \cite{Wi1, Wi2} and Yoshiki Kuramoto \cite{Ku} in 1975. Since then, synchronization has been extensively investigated in diverse scientific disciplines such as an applied mathematics, neuroscience and statistical physics, etc. We refer to survey articles and books \cite{A-B, A-B-F-H-K, D-B2, H-K-P-Z, P-R-K, Pe, St} for a brief introduction to the subject.  To fix the idea, we restrict our discussion to Kuramoto oscillators whose dynamics is governed by the sum of sinusoidal coupling of phase differences. 

Consider a lattice $\Lambda \subset \mathbb{R}^d$ with $N$ lattice points (or nodes), and we assume that Kuramoto oscillators are stationed on each lattice point, and interactions are all-to-all  with a uniform strength $\frac{\kappa}{N}$. To set up stage, let $\theta_i = \theta_i(t)$ be the phase of the Kuramoto oscillator at the $i$-th lattice point. In this setting, the phase dynamics is governed by the Cauchy problem to the (finite) Kuramoto model:
\begin{equation} \label{Ku}
	\begin{cases}
		\displaystyle {\dot \theta}_{i} = \nu_i +  \sum_{j\in[N]} \frac{\kappa}{N} \sin(\theta_j - \theta_i), \quad t > 0, \\
		\displaystyle \theta_i(0) = \theta_i^{\text{in}}, \quad  i \in [N]:= \{1, \ldots, N \},
	\end{cases}
\end{equation}
where $\nu_i$ is the natural frequency of the $i$-th oscillator.  Since the right-hand side of \eqref{Ku} is uniformly bounded and Lipschitz continuous, the standard Cauchy-Lipschitz theory guarantees a global well-posedness of smooth solutions. Thus, what matters for \eqref{Ku} lies on the emergent dynamics. In fact, the emergent dynamics of \eqref{Ku} has been extensively studied in literature, to name a few \cite{ B-C-M, B-D-P, C-S, D-X, D-X2, D-B1, D-B2, H-N-P, H-R, L-X, V-M0, V-M1, V-M2} from diverse scientific disciplines in last decades. In this paper, we are interested in the Kuramoto dynamics on the infinitely extended lattice, i.e., the number of Kuramoto oscillators is equivalent to the cardinality of the natural numbers. More specifically, we address the following set of questions: \newline
\begin{itemize}
	\item
	What is the suitable system describing the dynamics of an infinite number of Kuramoto oscillators?
	\vspace{0.2cm}
	\item
	If such a dynamical system exists, under what conditions on system parameters and initial data, can we rigorously show the emergent collective dynamics?
\end{itemize}
\vspace{0.2cm}
The main purpose of this paper is to answer the above proposed questions. In collective dynamics community, we often approximate infinite systems with all-to-all couplings by corresponding Vlasov type equations (see \cite{La}) which arise from large $N$-oscillator limit. In this way, mean-field approach can give approximate results for infinite system under consideration. Therefore, to get the exact result on the dynamics of infinite set of Kuramoto oscillators, we are forced to study the infinite set of ordinary differential equations as it is. In this regard, we propose the following natural extension of the finite model \eqref{Ku}:
\begin{equation}\label{IKM}
	\dot{\theta}_{i}=\nu_i+\sum_{j\in\mathbb{N}} \kappa_{ij}\sin\left(\theta_j-\theta_i \right), \quad t>0, \quad \forall~ i \in  \mathbb{N}, 
\end{equation}
where $\kappa_{ij}$ is the coupling strength between the $i$-th and $j$-th oscillators satisfying nonnegativity and row-summability:
\begin{equation}\label{A-1}
	K = (\kappa_{ij}), \quad \kappa_{ij} \geq 0,\quad  \| K \|_{\infty, 1} := \sup_{i\in\mathbb{N}} \sum_{j\in\mathbb{N}} \kappa_{ij}  <\infty.
\end{equation}
Note that without the coupling strength $\kappa_{ij}$, the infinite sum in the R.H.S. of \eqref{IKM} will not be well-defined. So introduction of such weight is needed. Moreover, unlike to the Kuramoto model, uniform coupling strength $\kappa_{ij} = \kappa$ does not satisfy the condition $\eqref{A-1}_3$. Of course, it is not complete new to study such an infinite set of ordinary differential equations. For example, coagulation and fragmentation process for polymer can be described by the infinite number of ODEs (see  \cite{B-C-P, Du, Sl}) and {infinite lattice Kuramoto model \cite{Br}}. Recently, Wang and Xue \cite{W-X} studied the flocking behaviors of the infinite number of Cucker-Smale particles, and they found that almost the same results in   \cite{Cu-S, H-Liu, H-T} for the original Cucker-Smale model can be obtained.  As first observed by the first author and his collaborators in \cite{H-L-R-S}, the first-order Kuramoto model can be lifted to the Cucker-Smale model by introducing auxiliary frequency variables. Thus, it is quite reasonable to study analogous study for the Kuramoto model without resorting on the corresponding mean-field equation.  The global well-posedness of \eqref{IKM} on the Banach space $(\ell^{\infty},~\| \cdot \|_{\infty})$ can be followed from the abstract Cauchy-Lipschitz theory together with the Lipschitz continuity of the R.H.S. of \eqref{IKM}  (see Proposition \ref{P2.2} and Lemma \ref{LA-1}). For the special situation:
\[ \kappa_{ij} \equiv 0,  \quad \theta_i  \equiv 0, \quad   \max\{i,j\} \geq N+1, \]
it is easy to see that the Kuramoto model \eqref{Ku} corresponds to the special case of the proposed infinite model \eqref{IKM}. Hence, whether the infinite system \eqref{IKM} can exhibit the emergent dynamics as in the Kuramoto model \eqref{Ku} or not will be a tempting question. Moreover, it would be very interesting to analyze distinct features which cannot be seen in the Kuramoto model with a finite system size. 

In what follows, we briefly discuss our main results documented in the following sections from Section \ref{sec:3} to Section \ref{sec:5}.  Let $\mathbb{N} = \{ 1, 2, \ldots \}$ be the set of all natural numbers. For the emergent dynamics of the infinite Kuramoto model \eqref{IKM} -- \eqref{A-1}, we consider two types of coupling gain matrix $K = (\kappa_{ij})$: 
\begin{align}
	\begin{aligned} \label{A-2}
		& \mbox{Row-summable network}:~\kappa_{ij} > 0,\quad i,j \in {\mathbb N},\quad \| K \|_{\infty. 1} < \infty,  \\
		& \mbox{Sender row-summable network}:~\kappa_{ij} = \kappa_{j}\geq 0,\quad  i,j \in {\mathbb N},\quad \| K \|_{\infty, 1} < \infty.   
	\end{aligned}
\end{align}
First, we consider positive and row-summability network topology $\eqref{A-2}_1$.  For a homogeneous ensemble with  the same natural frequency $\nu_i = \nu$, thanks to translational invariance property of \eqref{IKM}, we may assume that the common natural frequency $\nu$ is zero and \eqref{IKM} reduces to 
\begin{equation}\label{A-3}
	\dot{\theta}_{i}= \sum_{j\in\mathbb{N}} \kappa_{ij}\sin\left(\theta_j-\theta_i \right), \quad t>0, \quad \forall~ i \in  \mathbb{N}.
\end{equation}
In this case, depending on suitable conditions for the network topology $K = (\kappa_{ij})$, the phase diameter can be constant (see Corollary \ref{C3.1} and Proposition \ref{L3.2}, respectively). In particular, we can find an explicit example of non-decreasing phase diameter for some class of coupling gain matrix $K$. This is certainly a novel feature of the infinite model which cannot be seen in a finite system (see also Remark \ref{R3.2} and Corollary \ref{C3.1}). As can be seen in Proposition \ref{P2.1}, a gradient flow formulation for \eqref{Ku} plays a key role in the rigorous verificaion of phase-locking for a generic initial data in a large coupling regime \cite{D-X2, H-K-R, H-R}. Likewise,  the infinite system \eqref{A-3} can also be written as a gradient flow on a Banach space $\ell^2$ with the potential $P$ (see Proposition \ref{P3.4}):
\[  P(\Theta) = \frac{1}{2}\sum_{i,j\in\mathbb{N}} \kappa_{ij} (1-\cos(\theta_i - \theta_j)).   \]
Although we cannot use the Lojasiewicz gradient inequality in \cite{H-J} as it is, we can still use $P$ as a Lyapunov functional to derive complete synchronization (see Theorem \ref{T3.1}): 
\begin{equation} \label{A-4}
	\lim_{t \to \infty} \sup_{i,j\in\mathbb{N}} |{\dot \theta}_i(t) - {\dot \theta}_j(t) | = 0.
\end{equation}
On the other hand, for a heterogeneous ensemble with distinct natural frequencies, we can obtain a practical synchronization under suitable conditions on the coupling gain matrix $K = (\kappa_{ij})$ (Theorem \ref{T4.1}):
\[
\limsup_{t\to\infty}  \sup_{i,j\in\mathbb{N}} |\theta_i(t) - \theta_j(t) | \le\sin^{-1} \Big( {\mathcal O}(1) \frac{ {\mathcal D}\left(\mathcal{V}\right)}{\| K \|_{-\infty, 1}} \Big).
\]
Unfortunately, the complete synchronization estimate \eqref{A-4} for a heterogeneous ensemble is not available yet.   \newline

Second, we consider a row-summable sender network topology $\eqref{A-2}_2$. In this case, the infinite Kuramoto model reads as 
\begin{equation}\label{A-5}
	\dot{\theta}_{i}= \nu_i + \sum_{j\in\mathbb{N}} \kappa_j\sin\left(\theta_j-\theta_{i}\right), \quad t>0, \quad i \in {\mathbb N}.
\end{equation}
Compared to the aforementioned symmetric and summable network topology, we have better controls on the emergent dynamics. For a homogeneous ensemble, there might be two possible asymptotic states (one-point phase synchrony or bi-cluster configuration).  More precisely, let $\Theta$ be a solution to \eqref{A-5} with asymptotic configuration $\Theta^\infty=(\theta_1^\infty,\theta_2^\infty,\ldots)$. Then, we have
\[ \theta_{i}^{\infty}\in\left\{ \theta_{0}\right\} \cup\left\{ \theta_{0}\pm \kappa_{i}\pi\ |\ i\in\mathbb{N}\right\} \cup\left\{ \theta_{0}\pm\left(1-\kappa_{i}\right)\pi\ |\ i\in\mathbb{N}\right\}, \]
where 
\[ \theta_0 :=  \sum_{i\in\mathbb{N}} \kappa_{i}\theta_{i}^{\text{in}}. \]
We refer to Theorem \ref{T5.1} and Corollary \ref{C5.2} for details.  On the other hand, for a heterogeneous ensemble, we can rewrite system \eqref{A-5} into the second-order model with row-dimensional initial data:
\begin{equation} \label{A-6}
	\begin{cases}
		\displaystyle	\dot{\theta}_{i}=\omega_{i},\quad t>0,\quad\forall~ i\in\mathbb{N},\\
		\displaystyle	\dot{\omega}_{i}=\sum_{j\in\mathbb{N}} \kappa_j\cos\left(\theta_i-\theta_j\right)\left(\omega_j-\omega_{i}\right), \\
		\displaystyle	\theta_{i}(0)=\theta_{i}^{\text{in}}\in\mathbb{R},\quad\omega_{i}(0)=\nu_{i}+\sum_{j\in\mathbb{N}} \kappa_j\sin\left(\theta_j^{\text{in}}-\theta_{i}^{\text{in}}\right),
	\end{cases}
\end{equation}
where 
\[ \Theta^{\text{in}} = (\theta_1^{\text{in}}, \theta_2^{\text{in}}, \ldots ) \in \ell^{\infty}, \quad {\mathcal V} = (\nu_1, \nu_2, \ldots) \in \ell^{\infty}. \]
We set 
\[  {\mathcal W} := (\omega_1, \omega_2,  \ldots) \quad \mbox{and} \quad  {\mathcal D}({\mathcal W}) : = \sup_{m, n} |\omega_m - \omega_n|.  \]
In this case, under some restricted class of initial phase configuration confined in a quarter arc, we can show that the frequency diameter ${\mathcal D}({\mathcal W})$ decays to zero exponentially fast (see Theorem \ref{T5.2}).  \newline

The rest of this paper is organized as follows. In Section \ref{sec:2},  we briefly review the emergent dynamics of the finite Kuramoto model and study basic properties of the infinite Kuramoto model such as conservation law, translational invariance and several a priori estimates.  In Section \ref{sec:3} and Section \ref{sec:4},  we study emergent dynamics of \eqref{IKM} with a symmetric and row-summable network  topologies. In Section \ref{sec:5}, we investigate the complete synchronization of the infinite Kuramoto model with a sender network topology.  Finally, Section \ref{sec:6} is devoted to a brief summary of main results and discussion on some remaining issues for a future work. 

\vspace{0.5cm}

\noindent {\bf Notation}:~
Throughout the paper, we write the phase configuration vector and natural frequency vector as
\begin{align*}
	\begin{aligned}
		& \Theta_N := (\theta_1,\ldots,\theta_N), \quad \Theta := (\theta_1,\theta_2,\ldots),\\ & {\mathcal V}_N := (\nu_1,\ldots,\nu_N), \quad {\mathcal V}: = (\nu_1,\nu_2,\ldots),
	\end{aligned}
\end{align*}
and  we denote the set $\{1,\ldots,N\}$ by $[N]$ for simplicity. For $A=(a_1,a_2,\ldots)\in \mathbb{R}^\mathbb{N}$ and $p\in[1,\infty]$,  we set 
\[
\| { A} \|_p := \begin{cases}
	\displaystyle \Big( \sum_{i\in\mathbb{N}} |a_i|^p \Big)^{\frac{1}{p}}, \quad & 1 \leq p < \infty, \\
	\displaystyle  \sup_{i\in\mathbb{N}} |a_i|, \quad & p = \infty,
\end{cases}
\]
and denote $\ell^p=\ell^p({\mathbb N})$ the collection of all sequences with a finite $p$-th power sum:
\[ \ell^p({\mathbb N}) : = \Big \{ A\in\mathbb{R}^\mathbb{N} :~\| A \|_p  < \infty \Big \}, \quad p\in[1,\infty]. \]
Similarly, for every infinite matrix $K=(\kappa_{ij})\in\mathbb{R}^{\mathbb{N}\times \mathbb{N}}$  and $1\leq p,q\leq \infty$, we set 
\[ \|K\|_{p,q}:=\begin{dcases}
	\displaystyle\left|\sum_{{i\in \mathbb{N}}}\|(\kappa_{ij})_j\|_q^p\right|^{\frac{1}{p}} & (1\leq p<\infty),\\
	\displaystyle\sup_{i\in\mathbb{N}}\|(\kappa_{ij})_j\|_q& (p=\infty),
\end{dcases} \] 
and denote 
\[\ell^{p,q}:=\left\{K=(\kappa_{ij}):\| K \|_{p,q}<\infty \right\}, \]
which also becomes a normed vector space of infinite matrices. Finally, for every real vectors $X_N$ and $X$ given by 
\begin{align*}
	X_N=(x_1,\ldots,x_N)\in\mathbb{R}^N,\quad X=(x_1,x_2,\ldots)\in\mathbb{R}^\mathbb{N},
\end{align*}
we denote the supremum of the difference between their elements by
\begin{align*}
	{\mathcal D}(X_N) := \max_{i,j\in[N]} | x_i  - x_j |, \quad  {\mathcal D}(X) := \sup_{i,j\in\mathbb{N}} | x_i  - x_j |,
\end{align*}
and call the diameter of $X_N$ and $X$, respectively.
\section{Preliminaries}\label{sec:2}
\setcounter{equation}{0}
In this section, we study basic properties of the Kuramoto model on static networks with finite and infinite nodes. 
\subsection{Kuramoto model for a finite ensemble} \label{sec:2.1}
Consider the Cauchy problem to the Kuramoto model with a finite system size \cite{C-H-J-K, H-H-K, H-L-X, D-X2}:
\begin{equation}  \label{B-1}
	\begin{cases} 
		\displaystyle  \dot{\theta}_{i}=\nu_{i}+ \sum_{j\in[N]} \kappa_{ij}\sin\left(\theta_j-\theta_{i}\right),\quad t>0,\\
		\displaystyle \theta_{i}(0)=\theta_{i}^{\text{in}}, \quad i\in [N],
	\end{cases}
\end{equation}
where  $\kappa_{ij}$ is a nonnegative symmetric constant which denotes the strength between the $i$-th and $j$-th oscillators:
\begin{equation} \label{B-1-1}
	\kappa_{ij} = \kappa_{ji} \geq 0, \quad i, j \in [N]. 
\end{equation}
First, we recall some terminologies on emergent dynamics in the following definition. \newline
\begin{definition}\label{D2.1}
	Let $\Theta_N$ be a solution to \eqref{B-1} -- \eqref{B-1-1}. 
	\begin{enumerate}
		\item The state $\Theta$ is phase-locked if the phase differences are constant in time:
		\[\theta_i(t)-\theta_j(t)\equiv\theta_{ij},\quad t \geq 0, \quad  i,j\in[N].\]
		\item The state $\Theta$ achieves asymptotic phase-locking if and only if 
		\[\exists~\theta_{ij}^\infty=\lim_{t\to\infty}  (\theta_i(t)-\theta_j(t)),\quad i,j\in [N]. \]
		\item The state $\Theta$ achieves complete synchronization if and only if 
		\[ \lim_{t\to\infty}  {\mathcal D}({\dot \Theta}_N(t)) = 0. \]
	\end{enumerate}
\end{definition}
Next, we study basic preliminaries for \eqref{B-1} on conservation law and emergent dynamics.  For this, we set 
\begin{equation} \label{B-1-2}
	{\mathcal C}(t) :=  \sum_{i\in[N]} \theta_i - t \sum_{i\in[N]} \nu_i, \quad t \geq 0. 
\end{equation}

\begin{proposition}\label{P2.1}
	\emph{\cite{H-K-R, H-R}}
	Let $\Theta_N=(\theta_1,\ldots,\theta_N)$ be a solution to \eqref{B-1} -- \eqref{B-1-1}. Then, the following assertions hold.
	\begin{enumerate}
		\item
		(Balanced law):~The functional ${\mathcal C}$ in \eqref{B-1-2} is conserved along the flow \eqref{B-1}.
		\[ {\mathcal C}(t) = {\mathcal C}(0), \quad t \geq 0. \]
		\item
		(A gradient flow formulation):~If we define a potential $P_N = P_N(\Theta_N)$:
		\begin{equation*}\label{B-1-2-1}
			P_N(\Theta_N) :=  -\sum_{l\in[N]} \nu_l  \theta_l  + \frac{1}{2} \sum_{k,l\in[N]} \kappa_{kl} (1 - \cos (\theta_{k}-\theta_{l} )),
		\end{equation*}
		system \eqref{B-1} -- \eqref{B-1-1} can be rewritten as a gradient flow:
		\[ \partial_t \Theta_N = -\nabla_{\Theta_N} P_N(\Theta_N), \quad t > 0.  \]
		\item
		Suppose that network topology, natural freqencies and initial data satisfy 
		\begin{equation*} \label{B-1-3}
			\kappa_{ij} = \frac{\kappa}{N}, \quad \sum_{i\in[N]} \nu_i = 0, \quad  { R_0 := \Big|\frac{1}{N} \sum_{k \in[N]} e^{{\mathrm i} \theta_{k}^{\text{in}}}  \Big| > 0,} \quad \kappa > \frac{1.6}{R_0^2} {\mathcal D}({\mathcal V}_N).
		\end{equation*}
		Then, there exists an equilibrium state $\Theta_N^{\infty} =(\theta_1^\infty,\ldots,\theta_N^\infty)$ such that 
		\[ \lim_{t \to \infty} \| \Theta_N(t) - \Theta_N^{\infty} \|_{\infty} = 0. \]
	\end{enumerate}
\end{proposition}
\begin{proof}
	\noindent (i)~We  take sum \eqref{B-1} over all  $i$ and use \eqref{B-1-1} to get 
	\[  \frac{d}{dt} \sum_{i\in[N]} \theta_i = \sum_{i\in[N]} \nu_{i} + \sum_{i,j\in[N]}\kappa_{ij}\sin (\theta_j-\theta_{i}) = \sum_{i\in[N]} \nu_{i}.  \]
	This yields the desired conservation law. \newline

	\noindent (ii)~For a fixed $i \in [N]$,  we rewrite the potential $P_N$ as 
	\begin{align*}
		\begin{aligned}
			P_N(\Theta_N) &= -\nu_i  \theta_i +  \frac{1}{2} \sum_{l\in[N]} \kappa_{il} (1- \cos(\theta_i - \theta_l)) +  \frac{1}{2} \sum_{l\in[N]} \kappa_{li} (1- \cos(\theta_l - \theta_i)) \\
			&- \sum_{l\in[N]-\{i\}} \nu_l \theta_l   + \frac{1}{2} \sum_{k,l\in[N]-\{i\}} \kappa_{kl} (1- \cos(\theta_k - \theta_l)).
		\end{aligned}
	\end{align*}
	Now, we differentiate the above relation with respect to $\theta_i$ to find 
	\begin{align*}
		\begin{aligned}
			\partial_{\theta_i} P_N(\Theta_N) &= -\nu_i +  \frac{1}{2} \sum_{l\in[N]}  \kappa_{il} \sin(\theta_i - \theta_l) -\frac{1}{2} \sum_{l\in[N]} \kappa_{li} \sin(\theta_l - \theta_i) \\
			&= -\nu_i  -  \sum_{l\in[N]} \Big( \frac{\kappa_{il} + \kappa_{li}}{2} \Big) \sin(\theta_l - \theta_i)  \\
			&= -\nu_i - \sum_{l\in[N]}  \kappa_{il}  \sin (\theta_l - \theta_i)  \quad \mbox{using the symmetry of $(\kappa_{ij})$} \\
			&=-{\dot \theta}_i.
		\end{aligned}
	\end{align*}
	This yields 
	\[ {\dot \Theta}_N = -\nabla_{\Theta_N} P_N(\Theta_N). \]
	\noindent (iii)~Detailed argument can be found in \cite{H-R}. Thus, we just sketch the main line of idea as follows. First, we show that the phase configuration is uniformly bounded in the sense that there exists a positive constant $\theta^\infty$ such that 
	\[  \sup_{0 \leq t < \infty} \| \Theta_N(t) \|_{\infty} \leq \theta^{\infty}. \]
	Then, motivated by gradient flow approach in \cite{D-X2}, the authors in \cite{H-L-X,H-K-R, H-R} also used the gradient flow formulation (ii) and the analyticity of potential to say that there exists an equilibrium $\Theta_N^{\infty}$ such that \[ \lim_{t \to \infty} \| \Theta_N(t) - \Theta_N^{\infty} \|_{\infty} = 0. \]
\end{proof}
\subsection{Kuramoto model for an infinite ensemble}\label{sec:2.2}
In this subsection, we present several basic properties of the Kuramoto model which concerns the dynamics of countably infinite number of oscillators, in short `infinite Kuramoto model'. 

Note that for the following simple modification:
\[ \sum_{j\in[N]}\kappa_{ij}\sin\left(\theta_j-\theta_{i}\right) \quad \Longrightarrow \quad \sum_{j\in\mathbb{N}}\kappa_{ij}\sin\left(\theta_j-\theta_{i}\right), \]
the infinite sum in the R.H.S. of \eqref{IKM} might not be well-defined, unless we impose some restrictive asymptotic vanishing conditions on the network topology $K=(\kappa_{ij})_{i,j\in\mathbb{N}}$.  Once the infinite sum becomes well-defined, we  can consider the Cauchy problem to the infinite Kuramoto model:
\begin{equation}  \label{B-2}
	\begin{cases} 
		\displaystyle  \dot{\theta}_{i}=\nu_{i}+\sum_{j\in\mathbb{N}}\kappa_{ij}\sin\left(\theta_j-\theta_{i}\right),\quad t>0, \\
		\displaystyle \theta_{i}(0)=\theta_{i}^{\text{in}}, \quad \quad i \in {\mathbb N},
	\end{cases}
\end{equation}
where $\Theta^{\text{in}}, \mathcal{V}$ and $K=(\kappa_{ij})$ satisfy
\begin{equation}  \label{B-3}
	\Theta^{\text{in}}\in\ell^p,\quad \mathcal{V}\in \ell^p,\quad K\in \ell^{p,1},\quad \kappa_{ij}\geq 0,\quad \forall~i,j\in\mathbb{N}
\end{equation}  
for some $p\in [1,\infty]$. Unlike in Section \ref{sec:2.1}, we allow the asymmetric network topology $K$ to consider the most general case.
Then, the following proposition guarantees the well-posedness of \eqref{B-2} -- \eqref{B-3} by using the standard Cauchy-Lipschitz theory.

\begin{proposition}\label{P2.2}
	Suppose that initial configuration, natural frequencies and network topology satisfy \eqref{B-3}. Then, there exists a unique smooth solution $\Theta = \Theta(t) \in {\mathcal C}^1(\mathbb{R}_+; \ell^p)$ to the infinite system
	\eqref{B-2}.
\end{proposition}
\begin{proof}  
	First of all, we set 
	\[
	f_i(\Theta) := \nu_i+\sum_{j\in\mathbb{N}} \kappa_{ij} \sin(\theta_j- \theta_i),  \quad i \in {\mathbb N}, \quad   {\mathcal F}(\Theta) := (f_1(\Theta),f_2(\Theta), \ldots ).
	\]
	In order to use the standard Cauchy-Lipschitz theory on the Banach space $\ell^{p}$, it suffices to show that for every two solutions $\Theta$ and ${\tilde\Theta}$ to { \eqref{B-2} -- \eqref{B-3}}, we have
	\begin{equation} \label{B-5}
		\| {\mathcal F} \|_{p} \leq \| {\mathcal V} \|_{p} + \| K \|_{p, 1}, \quad  \left\| {\mathcal F}(\Theta)- {\mathcal F}(\widetilde{\Theta})\right\|_p \leq  2 \| K \|_{p, 1}  \|\Theta - {\tilde \Theta} \|_{p}.
	\end{equation}
	\noindent $\bullet$~(Derivation of $\eqref{B-5}_1$):~~For each $i\in\mathbb{N}$, we have
	\begin{equation} \label{B-6}
		|f_i(\Theta)|=\left|\nu_i+\sum_{j\in\mathbb{N}} \kappa_{ij}\sin(\theta_j-\theta_i) \right|\leq |\nu_i|+ \sum_{j\in\mathbb{N}} \kappa_{ij}. 
	\end{equation}
	Then by the Minkowski inequality and \eqref{B-6}, we have 
	\[\| {\mathcal F} \|_{p} \leq \| {\mathcal V} \|_{p} + \| K \|_{p, 1}\quad\mbox{for}~~ 1\leq p\leq\infty.\]
	
	\noindent $\bullet$~(Derivation of $\eqref{B-5}_2$):~For $1\leq p<\infty$, every $\Theta, {\widetilde \Theta} \in \ell^{p}$ satisfy
	\begin{equation}\label{B-7}
		\begin{aligned} 
			\|\mathcal{F}(\Theta)-\mathcal{F}(\widetilde{\Theta})\|_p^p&=\sum_{i\in\mathbb{N}}|f_i(\Theta)-f_i(\widetilde{\Theta})|^p\\
			&=\sum_{i\in\mathbb{N}}\left|\sum_{j\in\mathbb{N}} \kappa_{ij}\left(\sin\left(\theta_{j}-\theta_{i} \right) - \sin ({\tilde \theta}_{j} - {\tilde \theta}_{i})\right)\right|^p \\
			&\le\sum_{i\in\mathbb{N}}\left|\sum_{j\in\mathbb{N}} \kappa_{ij}\left|\left(\theta_{j} -\theta_{i} \right)- ({\tilde \theta}_{j} - {\tilde \theta}_{i})\right|\right|^p\\
			&\le\sum_{i\in\mathbb{N}}\left|\sum_{j\in\mathbb{N}} \kappa_{ij}|\theta_j-{\tilde \theta}_j|+|\theta_i-{\tilde\theta}_i|\sum_{j =1}^\infty \kappa_{ij}\right|^p\\
			&\leq 2^{p-1}\sum_{i\in\mathbb{N}} \left[\Big(\sum_{j\in\mathbb{N}} \kappa_{ij}|\theta_j-{\tilde \theta}_j|\Big)^p+\Big(|\theta_i-{\tilde\theta}_i|\sum_{j\in\mathbb{N}} \kappa_{ij}\Big)^p\right]\\
			&\leq 2^{p-1}\sum_{i\in\mathbb{N}}\left[\|(\kappa_{ij})_j\|_q^p\left(\sum_{k\in \mathbb{N}} |\theta_k-\tilde{\theta}_k|^p\right) + |\theta_i-\tilde{\theta}_i|^p\Big(\sum_{j\in\mathbb{N}} \kappa_{ij}\Big)^p \right],
		\end{aligned}
	\end{equation}
	where we used the H\"{o}lder inequality for $q=\frac{p}{p-1}$ in the last inequality. If we apply the following relations
	\[\|X\|_q\leq \|X\|_1,\quad \sum_{i\in\mathbb{N}} |y_iz_i|\leq \Big(\sum_{i\in\mathbb{N}} |y_i|\Big)\Big(\sum_{i\in\mathbb{N}} |z_i|\Big),\quad \forall~X, Y, Z\in \mathbb{R}^{\mathbb{N}} \]
	to \eqref{B-7}, then we have desired estimate:
	\begin{equation*}
		\begin{aligned}
			\|\mathcal{F}(\Theta)-\mathcal{F}(\widetilde{\Theta})\|_p^p&\leq  2^{p-1}\sum_{i\in\mathbb{N}}\left[\|(\kappa_{ij})_j\|_q^p\left(\sum_{k\in\mathbb{N}} |\theta_k-\tilde{\theta}_k|^p\right)+|\theta_i-\tilde{\theta}_i|^p\Big(\sum_{j\in\mathbb{N} } \kappa_{ij}\Big)^p \right]\\
			&\leq 2^{p-1} \left[\|K\|_{p,1}^p\|\Theta-\widetilde{\Theta}\|_p^p+\|\Theta-\widetilde{\Theta}\|_p^p\|K\|_{p,1}^p \right]\\
			&=2^p\|K\|_{p,1}^p\|\Theta-\widetilde{\Theta}\|_p^p.
		\end{aligned}
	\end{equation*}
	In addition, we can also obtain desired estimate for $p=\infty$:
	\begin{align*}
		\left\| {\mathcal F}(\Theta)- {\mathcal F}(\widetilde{\Theta})\right\|_\infty 
		&=\sup_{i\in\mathbb{N}}|f_i(\Theta)-f_i(\widetilde{\Theta})|\\
		&=\sup_{i\in\mathbb{N}}\left|\sum_{j\in\mathbb{N}} \kappa_{ij}\left(\sin\left(\theta_{j}-\theta_{i} \right) - \sin ({\tilde \theta}_{j} - {\tilde \theta}_{i})\right)\right|\\
		&\le\sup_{i\in\mathbb{N}} \sum_{j\in\mathbb{N}} \kappa_{ij}\left|\left(\theta_{j} -\theta_{i} \right)- ({\tilde \theta}_{j} - {\tilde \theta}_{i})\right|\\
		& \le\sup_{i\in\mathbb{N}}\sum_{j\in\mathbb{N}} \kappa_{ij}\left( |\theta_{j} - {\tilde \theta}_{j} |+  |\theta_{i} - {\tilde \theta}_{i}  |\right) \\
		& \le 2 \Big( \sup_{i\in\mathbb{N}} \sum_{j\in\mathbb{N}} \kappa_{ij} \Big)  \| \Theta - {\widetilde \Theta} \|_\infty  \\
		&= 2 \| K \|_{\infty, 1}  \|\Theta - {\tilde \Theta} \|_{\infty}.
	\end{align*}  
	
	Now, once we have \eqref{B-5}, the solution to \eqref{B-2} -- \eqref{B-3} exists uniquely in some nonempty finite time interval $[0,T]$, and the solution never blows up in finite time due to the boundedness of the image of ${\mathcal F}$ so that the local solution can be extended to the global solution $\Theta:[0,\infty)\to \ell^p.$
\end{proof}

In the following lemma, we can see the analogous properties of our infinite model with the finite Kuramoto model. Lemma \ref{L2.1} (1) gives an invariant of our model, and (2) gives the translation-invariant property of the Kuramoto model. Then in Lemma \ref{L2.2}, we discuss two basic sets of estimates to be used in Section \ref{sec:3.1}, \ref{sec:4} and \ref{sec:5.2}, then we establish Lipschitz continuity of some functionals.
\begin{lemma}\label{L2.1}
	Let $p,q\in[1,\infty]$ with $\frac{1}{p}+\frac{1}{q}=1$, and let $\Theta$ be a global $\ell^p$-solution to \eqref{B-2} -- \eqref{B-3}. Then, the following assertions hold.
	\begin{enumerate}
		\item	
		If the network topology $K=(\kappa_{ij})$ is given by 		
		\begin{equation}\label{B-8}
			\kappa_{ij}=a_{ij}\kappa_j,\quad \forall~i,j\in\mathbb{N},
		\end{equation}
		for some symmetric $A=(a_{ij})\in \ell^{p,p}$ and $(\kappa_1,\kappa_2,\ldots)\in\ell^{q}$,	we have 
		\[\frac{d}{dt}\left(\sum_{i\in\mathbb{N}}  \kappa_i \theta_i \right)=\sum_{i\in\mathbb{N}} \kappa_i\nu_i. \]
		\item
		If we set 
		\begin{equation*}
			{\hat \theta}_i(t) = \theta_i(t)-\nu t, \quad i \in {\mathbb N},\quad t \geq 0,
		\end{equation*}
		then ${\hat \Theta} := ({\hat \theta}_1,{\hat \theta}_2, \ldots)$ satisfies 
		\begin{equation*}
			\begin{cases}
				\displaystyle\dot{\hat{\theta}}_{i}=\nu_i-\nu+\sum_{j\in\mathbb{N}}\kappa_{ij}\sin (\hat{\theta}_{j}-\hat{\theta}_{i}), \quad t>0,\\
				\displaystyle\hat{\theta}_i(0)=\theta_i^{\text{in}}\in \mathbb{R},\quad \forall~ i\in \mathbb{N}.
			\end{cases}
		\end{equation*}
	\end{enumerate}		
\end{lemma}
\begin{proof} (1)~First, we multiply $\kappa_i$  to $\eqref{B-2}_1$ to obtain
	\begin{equation} \label{B-9}
		\frac{d}{dt} \Big( \kappa_i \theta_{i} \Big) = \kappa_i \nu_{i}+\sum_{j\in\mathbb{N}} \kappa_i \kappa_{ij}\sin\left(\theta_{j}-\theta_{i}\right).
	\end{equation}
	Then, we take a summation of \eqref{B-9} over all $i$ and use the exchange symmetry $i \longleftrightarrow j$  to get 
	{
		\begin{equation*}
			\frac{d}{dt} \sum_{i\in\mathbb{N}} \kappa_i \theta_{i} = \sum_{i\in\mathbb{N}} \left(\kappa_i \nu_{i}+\sum_{j\in\mathbb{N}} \kappa_i \kappa_{ij}\sin\left(\theta_{j}-\theta_{i}\right) \right)
			= \sum_{i\in\mathbb{N}} \left( \kappa_i \nu_{i} - \sum_{j\in\mathbb{N}} \kappa_j \kappa_{ji}\sin\left(\theta_{j}-\theta_{i}\right)\right).
		\end{equation*}
	}
	Therefore, we employ \eqref{B-8} to get the desired balanced law. \newline
	
	\noindent (2)~Since the second assertion is obvious, we omit its proof. 
\end{proof}

\begin{remark}\label{R2.1} (1)~If we set $p=1$ and $\kappa_{j}\equiv 1$, then the network topology $K$ satisfying \eqref{B-8} is a symmetric summable infinite matrix (see Section \ref{sec:3}):
	\[K=(\kappa_{ij})\in\ell^{1,1},\quad \kappa_{ij}=\kappa_{ji},\quad i,j\in\mathbb{N}. \] 
	(2)~If we set $p=\infty$ and $a_{ij}\equiv 1$, then the network topology $K$ satisfying \eqref{B-8} is a sender network (see Section \ref{sec:5}):
	\[(\kappa_1,\kappa_2,\ldots)\in\ell^1,\quad \kappa_{ij}=\kappa_j,\quad i,j\in\mathbb{N}. \]
\end{remark}

\begin{lemma}\label{L2.2}
	Let $\Theta = \Theta(t)$ be a global $\ell^\infty$-solution to \eqref{B-2} -- \eqref{B-3}. Then, the following assertions hold.
	\begin{enumerate}
		\item
		$\dot{\Theta}$ and $\ddot{\Theta}$ are uniformly bounded: for every $i\in\mathbb{N}$,
		\begin{equation} \label{B-10}
			\begin{aligned}
				&\sup_{0 \leq t  < \infty}  |\dot{\theta}_i(t)|  \leq \|\mathcal{V}\|_\infty+ \| K \|_{\infty,1} \leq \|\mathcal{V}\|_p+ \| K \|_{p,1},\\
				&\sup_{0 \leq t < \infty} |\ddot{\theta}_i(t)|  \leq 2 \| K \|_{\infty, 1} (\|\mathcal{V}\|_\infty +  \| K \|_{\infty, 1})\leq 2 \| K \|_{p, 1} (\|\mathcal{V}\|_p +  \| K \|_{p, 1}).
			\end{aligned}
		\end{equation}
		\item 
		Extremals and phase-diameter functionals 
		\[ {\mathcal D}(\Theta) := \sup_{i,j\in \mathbb{N}} |\theta_i - \theta_j|, \quad  \sup_{i\in \mathbb{N}}\theta_i \quad \mbox{and} \quad  \inf_{i \in \mathbb{N} }\theta_i \]
		are Lipschitz continuous in time $t$. 
	\end{enumerate}
\end{lemma}
\begin{proof}
	\noindent	(1)  The first estimate follows from \eqref{B-6}. Now, we differentiate $\eqref{B-2}_1$ with respect to $t$ and use $\eqref{B-10}_1$ to obtain 
	\begin{align}
		\begin{aligned} \label{B-11}
			\left|\ddot{\theta}_{i}\right|	&=\left|\sum_{j\in\mathbb{N}} \kappa_{ij} (\dot{\theta}_{i}-\dot{\theta}_{j} )\cos\left(\theta_{i}-\theta_{j}\right)\right|\\
			& \le2\left(\left\Vert \mathcal{V}\right\Vert _{\infty}+   \| K \|_{\infty, 1}  \right)\sum_{j\in\mathbb{N}} \kappa_{ij} \\
			&\le 2  \| K \|_{\infty, 1} \left(\left\Vert \mathcal{V}\right\Vert _{\infty}+  \| K \|_{\infty, 1} \right)\\
			&\leq 2 \| K \|_{p, 1} (\|\mathcal{V}\|_p +  \| K \|_{p, 1}).
		\end{aligned}
	\end{align}
	
	%\noindent (2)~Next, we will show that the diameter functional ${\mathcal D}(\Theta)$ is globally Lipschitz continuous with  the Lipschitz constant $L = D\left(\mathcal{V}\right)+2 \| A \|_{\infty, 1}$. For $i, j \in {\mathbb N}$, note that 
	%	\begin{align*}
		%	\begin{aligned}
			%		\Big | {\dot \theta}_i - {\dot \theta}_j \Big| & \le\left|\nu_{i}-\nu_{j}\right|+\left|\sum_{k\in\mathbb{N}}\left(a_{ik}\sin\left(\theta_{k}-\theta_{i}\right)+a_{jk}\sin\left(\theta_{k}-\theta_{j}\right)\right)\right|\\
			%		& \le {\mathcal D}(\mathcal{V})+\sum_{k\in\mathbb{N}} (a_{ik}+a_{jk}) \le {\mathcal D}(\mathcal{V})+2  \| A \|_{\infty,1}.
			%	\end{aligned}
		%	\end{align*}
	%This implies 
	%\[ \Big | \theta_{i}(t+s)-\theta_{j}(t+s) - (\theta_{i}(t) -\theta_{j}(t) ) \Big | = | ({\dot \theta}_{i} - {\dot \theta}_{j})(\xi_{t,s})| |s| \leq ( {\mathcal D}(\mathcal{V})+2  \| A \|_{\infty,1} ) |s|, \]
	%where $\xi_{t,s}$ lies between $t$ and $t+s$. This yields
	%\[  |{\mathcal D}(\Theta(t+s)) - {\mathcal D}(\Theta(t))| \leq  ( {\mathcal D}(\mathcal{V})+2  \| A \|_{\infty,1} ) |s|.  \]
	\noindent (2)~We first consider the Lipschitz continuity of \[ t\mapsto \sup_{i\in \mathbb{N}}  \theta_i(t).\] For every $s < t$,  we use  Lemma \ref{L2.2} (1) to get 
	\[ \theta_i(t) \leq \theta_i(s) + (\| {\mathcal V} \|_\infty + \| K \|_{\infty, 1}) (t-s) \leq \sup_{i\in \mathbb{N}}  \theta_i(s) + (\| {\mathcal V} \|_\infty + \| K \|_{\infty, 1}) (t-s).  \]
	Then, we take the supremum of the L.H.S. of the above relation to obtain
	\begin{equation} \label{B-11-1}
		\sup_{i\in \mathbb{N}}  \theta_i(t) \leq \sup_{i\in \mathbb{N}}  \theta_i(s) + (\| {\mathcal V} \|_\infty + \| K \|_{\infty, 1}) (t-s),
	\end{equation}
	and a similar argument also yields
	\begin{equation} \label{B-11-2}
		\sup_{i\in \mathbb{N}}  \theta_i(t) \geq \sup_{i\in \mathbb{N}}  \theta_i(s) - (\| {\mathcal V} \|_\infty + \| K \|_{\infty, 1}) (t-s).
	\end{equation}
	Therefore, we combine \eqref{B-11-1} and \eqref{B-11-2} to obtain
	\[ \Big |  \sup_{i\in \mathbb{N}}  \theta_i(t) - \sup_{i\in \mathbb{N}}  \theta_i(s) \Big| \leq  (\| {\mathcal V} \|_\infty + \| K \|_{\infty, 1}) |t-s|,\quad \forall~t,s\geq 0. \]
	In addtion, the Lipschitz continuity of 
	\[t\mapsto \inf_{i\in \mathbb{N}} \theta_i(t)\] 
	can be also shown in a similar manner. Finally, the phase-diameter $\mathcal{D}(\Theta)$, which can be given by the difference between these two extremals, is also Lipschitz.
\end{proof}
\begin{remark} \label{R2.2}
	Note that the relation \eqref{B-11} yields
	\begin{equation} \label{B-12}
		\sup_{0 \leq t < \infty} \left|\ddot{\theta}_{i}(t) \right| \leq 2\left(\left\Vert \mathcal{V}\right\Vert _{\infty}+   \| K \|_{\infty, 1}  \right)\sum_{j\in\mathbb{N}}\kappa_{ij},\quad\forall~ i\in \mathbb{N}.
	\end{equation}
\end{remark}

\begin{lemma}\label{L2.3}
	Let $\Theta = \Theta(t)$ be a global solution to \eqref{B-2} -- \eqref{B-3}. Then  for every $i, j \in {\mathbb N}$, we have
	\[ \left|\frac{d}{dt}\left({\theta}_{i}-{\theta}_{j}\right)\right|\leq \mathcal{D}(\mathcal{V})+2 \| K \|_{\infty, 1},\quad\left|\frac{d^{2}}{dt^{2}}\left({\theta}_{i}-{\theta}_{j}\right)\right|\leq2  \| K \|_{\infty, 1} \left({\mathcal D}\left({\mathcal V}\right)+2 \| K \|_{\infty, 1} \right). \]
\end{lemma}
\vspace{0.1cm}
\begin{proof} 
	For every $i,j\in\mathbb{N}$, the first and the second derivatives of $\theta_i-\theta_j$ are given by
	\begin{equation} \label{B-13}
		\begin{aligned}
			\frac{d}{dt}\left({\theta}_{i}-{\theta}_{j}\right) &=\nu_{i}-\nu_{j}-\sum_{k\in\mathbb{N}} \left[\kappa_{ik}\sin(\theta_{i}-\theta_{k})+\kappa_{jk}\sin(\theta_{k}-\theta_{j})\right],\\
			\frac{d^{2}}{dt^{2}}\left({\theta}_{i}-{\theta}_{j}\right) &=-\sum_{k\in\mathbb{N}} \left[\kappa_{ik}\cos(\theta_{i}-\theta_{k})\frac{d}{dt}\left({\theta}_{i}-{\theta}_{k}\right)+\kappa_{jk}\cos(\theta_{k}-\theta_{j})\frac{d}{dt}\left({\theta}_{k}-{\theta}_{j}\right)\right].
		\end{aligned}
	\end{equation}
	Then, we have the boundedness of $\frac{d}{dt}\left({\theta}_{i}-{\theta}_{j}\right)$ from the following inequalities:
	\begin{equation}\label{B-14}
		\Big| \frac{d}{dt}\left({\theta}_{i}-{\theta}_{j}\right) \Big| \leq {\mathcal D}({\mathcal V}) + \sum_{k\in\mathbb{N}} (\kappa_{ik} + \kappa_{jk})  \leq {\mathcal D}({\mathcal V})  + 2 \| K \|_{\infty,1}.
	\end{equation}
	Finally, we combine \eqref{B-13} and \eqref{B-14} to obtain the boundedness of $\frac{d^2}{dt^2}\left({\theta}_{i}-{\theta}_{j}\right)$:
	\[
	\left|\frac{d^2}{dt^2}\left({\theta}_{i}-{\theta}_{j}\right)\right| 
	\leq  \sum_{k\in\mathbb{N}} \Big[ \kappa_{ik} \Big| \frac{d}{dt}\left({\theta}_{i}-{\theta}_{k}\right) \Big| + \kappa_{jk} \Big| \frac{d}{dt}\left({\theta}_{k}-{\theta}_{j}\right) \Big| \Big]  \leq 2 \| K \|_{\infty, 1} \left(\mathcal{D}\left({\mathcal V}\right)+2 \|K\|_{\infty,1}\right).
	\]
\end{proof}

\section{Emergent dynamics of a homogeneous ensemble} \label{sec:3}
\setcounter{equation}{0}
In this section, we present the emergent dynamics of the infinite Kuramoto model with a homogeneous ensemble consisting of oscillators with identical natural frequencies:
\[
\quad\nu_{i}\equiv \nu,\quad \forall~i\in \mathbb{N}.
\]
From Lemma \ref{L2.1} (2), we may assume that $\nu_i \equiv 0$ without loss of generality. In other words, we consider the phase configuration $\Theta$ satisfying
\begin{equation} 
	\begin{cases} \label{C-1}
		\displaystyle\dot{\theta}_{i}=\sum_{j\in\mathbb{N}} \kappa_{ij}\sin\left(\theta_{j}-\theta_{i}\right), \quad t>0,\\
		\displaystyle\theta_i(0)=\theta_i^{\text{in}}\in \mathbb{R},\quad \forall~ i\in \mathbb{N},\\
		\Theta^{\text{in}}=(\theta_1^{\text{in}},\theta_2^{\text{in}},\ldots)\in\ell^p,\quad K=(\kappa_{ij})\in\ell^{p,1},\quad p\in[1,\infty].
	\end{cases}
\end{equation}
Since $\ell^1\subset \ell^2\subset\cdots\subset \ell^\infty$ and $\ell^{1,1}\subset \ell^{2,1}\subset\cdots\subset \ell^{\infty,1}$, all results for $\ell^p$-solution $\Theta$ can also be applied to other $\ell^q$-solutions with $q<p$. In the sequel, we will study the dynamics of $\ell^\infty$-solution and $\ell^p$-solution ($p<\infty$) to \eqref{C-1}  and provide some results corresponding to each part of Proposition \ref{P2.1}.
\medskip
\subsection{$\ell^\infty$-solution: complete synchronization} \label{sec:3.1}
In this subsection, we will study the complete synchronization of the homoegeneous ensemble with $\ell^\infty$ initial data. As aforementioned, all results in this subsection can be applied to other $\ell^p$-solutions.
\subsubsection{Dynamics of phase diameter}
At a heuristic level, it is natural to expect that ${\mathcal D}(\Theta(t))$ is `{\it non-increasing}' in $t$ whenever ${\mathcal D}(\Theta(t))<\pi$, since the oscillators near the extremal phases 

\[ {\overline \theta}(t) :=\sup_{i\in \mathbb{N}}\theta_i(t),\quad \underline{\theta}(t) :=\inf_{i\in\mathbb{N}}\theta_i(t), \]
are pulled inward the region in which the majority of the group is located. 
In fact, for the finite Kuramoto ensemble, it is easy to check that ${\overline \theta}$ and ${\underline \theta}$ are nonincreasing and nondecreasing, respectively, and their difference converges to zero exponentially, so that Proposition \ref{P2.1} (3) holds. For the infinite Kuramoto ensemble, however, a such argument has to be refined. The following lemma shows that such a heuristic argument holds, when the interaction network $K = (\kappa_{ij})$ satisfies some structural condition uniformly in $i$.\\

Throughout the paper, we refer to the following frameworks to guarantee the synchronization behavior of the infinite Kuramoto model: \newline
\begin{itemize}
	\item
	(${\mathcal F}1$): The initial phase-diameter is smaller than $\pi$: the initial phase configuration $\Theta^{\text {in}}$ satisfies
	\[\mathcal{D}(\Theta^{\text{in}})<\pi. \]
	
	\vspace{0.2cm}
	
	\item
	(${\mathcal F}2$): There exists a sequence $\boldsymbol{\tilde\kappa} := \{\tilde\kappa_j\}_{j\in \mathbb{N}}\in \ell^1$ such that 
	\[\frac{\kappa_{ij}}{\sum_{k\in\mathbb{N}}\kappa_{ik}}>\tilde\kappa_j>0,\]
	for all $i,j\in \mathbb{N}.$
\end{itemize}

\begin{lemma} \label{L3.1}
	Suppose that network topology $K=(\kappa_{ij})$ and initial data $\Theta^{\text{in}}$ satisfy $({\mathcal F}1)-({\mathcal F}2)$, and let $\Theta$ be a solution to \eqref{C-1} with $p=\infty$. If initial phase diameter $\mathcal{D}(\Theta^{\text{in}})$ is nonzero, there exist two positive constants $\delta$ and $\varepsilon$ such that 
	\begin{itemize}
		\item
		For every index $i \in {\mathbb N}$ satisfying  $ \theta_i^{\text{in}}\leq  {\overline \theta}(0) - \varepsilon$,  one has 
		\[   \theta_i(t)< {\overline \theta}(0),\quad \forall~t\in(0, \delta).  \]
		\item
		For every index $i \in {\mathbb N}$  satisfying $\theta_i^{\text{in}}>   {\overline \theta}(0)-\varepsilon$,  one has 
		\[  \dot{\theta}_i(t) < 0,\quad \forall~t\in(0,\delta). \]
	\end{itemize}
\end{lemma}
\begin{proof} Since the proof is very lengthy and technical, we leave its proof in Appendix \ref{App-B}. 
\end{proof}
\vspace{0.3cm}
Note that the natural frequency $\mathcal{V}$ and initial phase diameter $\mathcal{D}(\Theta_N^{\text{in}})$ satisfy the same condition in Proposition \ref{P2.1} (3), and only the positivity condition $\kappa_{ij}>0$ has been modified to $(\mathcal{F}2)$. 
\vspace{0.3cm}
\begin{remark} \label{R3.1}
	Below, we provide several remarks on the framework $(\mathcal{F}1)-(\mathcal{F}2)$.
	\begin{enumerate}
		\item
		An interaction network $(\kappa_{ij})$ satisfying $(\mathcal{F}2)$ can be easily constructed from a sequence in $\ell^1$ whose components are all positive real numbers. More precisely, for a positive sequence $\{a_{i}\}_{i\in\mathbb{N}} \in  \ell^1$, we set 
		\[ \kappa_{ij}:= a_i a_j, \quad \forall~i, j \in {\mathbb N}. \]
		Then, the framework $(\mathcal{F}2)$ holds true by the following relation:
		\[ \displaystyle \frac{\kappa_{ij}}{\sum_{k\in\mathbb{N}} \kappa_{ik}} = \frac{a_j}{\sum_{k\in\mathbb{N}} a_k} >  \frac{a_j}{\sum_{k\in\mathbb{N}} a_k + 1 } =: \tilde\kappa_j. \]
		\vspace{0.1cm}
		\item
		The sequence $\boldsymbol{\tilde \kappa}=\{\tilde\kappa_j\}_{j\in\mathbb{N}}$ in $(\mathcal{F}2)$ is always contained in $\ell^1$. In fact, its $\ell^1$-norm is always smaller than $1$.
		\vspace{0.1cm}
		\item
		For the trivial initial data with ${\mathcal D}(\Theta^{\text{in}}) = 0$, we have
		\[  
		\dot{\theta}_i(t) = 0, \quad \forall~i\in\mathbb{N},\quad t > 0.
		\]
		Thus, the solution $\Theta$ is a steady state solution where whole phases are concentrated in a  singleton.
		
		\vspace{0.1cm}
		
	\end{enumerate}
\end{remark}	
\vspace{0.3cm}
As a consequence of Lemma \ref{L3.1}, one can see that the phase diameter $\mathcal{D}(\Theta)$ is also nonincreasing in $t$ as in Proposition \ref{P2.1}, though we do not have any estimate on the decay rate yet. 
\vspace{0.3cm}

\begin{corollary} \label{C3.1}
	Suppose that network topology $K=(\kappa_{ij})$ and initial data $\Theta^{\text{in}}$ satisfy $({\mathcal F}1)-({\mathcal F}2)$, and let $\Theta=\theta(t)$ be a solution to \eqref{C-1} with $p=\infty$. Then, the following assertions hold:	
	\begin{enumerate}
		\item The phase-diameter ${\mathcal D}(\Theta(t))$ is non-increasing  $t$:
		\[ {\mathcal D}(\Theta(t))\leq {\mathcal D}(\Theta^{\text{in}})<\pi,\quad \forall~t > 0. \]
		\item  $\Theta$ is a phase-locked state if and only if
		\[ \mathcal{D}(\Theta)\equiv 0. \]
	\end{enumerate}
\end{corollary}
\begin{proof}  (1)~We split the proof into two cases:
	\[ {\mathcal D}(\Theta^{\text {in}}) = 0, \quad  0 < {\mathcal D}(\Theta^{\text{in}}) < \pi. \]
	\vspace{0.1cm}
	
	\noindent $\diamond$~Case A $(  {\mathcal D}(\Theta^{\text{in}}) = 0)$: In this case, as discussed in Remark \ref{R3.1}(3), we have
	\[  {\mathcal D}(\Theta(t)) =  {\mathcal D}(\Theta^{\text{in}}) = 0, \quad t > 0, \]
	which yields the desired result. 
	
	\vspace{0.1cm}
	
	\noindent $\diamond$~Case B $(0 <  {\mathcal D}(\Theta^{\text{in}}) < \pi)$: 
	From Lemma \ref{L2.2}(2), the set 
	\[\left\{t\geq 0: \mathcal{D}(\Theta(t))\leq \mathcal{D}(\Theta^{\text{in}}) \right\} \]
	is closed in $[0,\infty)$. On the other hand, Lemma \ref{L3.1} implies that it is also a nonempty open subset of $[0,\infty)$. Therefore, we have 
	\[\left\{t\geq 0: \mathcal{D}(\Theta(t))\leq \mathcal{D}(\Theta^{\text{in}}) \right\}=[0,\infty), \]
	which is our desired result.\\
	
	\vspace{0.0cm} 
	
	\noindent (2)~It is sufficient to prove the `only if' part. If $\mathcal{D}(\Theta(t_0))>0$, then Lemma \ref{L3.1} can be applied, so that there exists a neighborhood of $\overline{\theta}(t_0)$ such that every $\theta_i$ in the neighborhood decreases strictly. Similarly, there exists a neighborhood of $\underline{\theta}(t_0)$ such that every $\theta_j$ in the neighborhood increases strictly. This contradicts the phase-locked assumption, as it is necessary to satisfy
	\[\dot{\theta}_i-\dot{\theta}_j=\frac{d}{dt}(\theta_i-\theta_j)=0, \quad i,j\in\mathbb{N}, \]
	for the phase-locked state $\Theta$.
\end{proof}
\vspace{0.3cm}
Note that Corollary \ref{C3.1} does not guarantee that the phase diameter is strictly decreasing. If $\overline{\theta}(t_0)$ is not a limit point of $\{\theta_i(t_0)\}_{i\in\mathbb{N}}$, then there exists a neighborhood $U$ of $\overline{\theta}(t_0)$ which contains only finitely many $\theta_i$'s, and the supremum $\overline{\theta}(t)$ is determined by those finitely many $\theta_i$'s for all $t$ sufficiently close to $t_0$. Therefore, Lemma \ref{L3.1} implies that the supremum $\overline{\theta}$ decreases strictly at time $t=t_0$ if $\overline{\theta}(t_0)$ is not a limit point of $\{\theta_i(t_0)\}_{i\in\mathbb{N}}$. However, if both $\overline{\theta}(t_0)$ and $\underline{\theta}(t_0)$ are the limit points of $\{\theta_i(t_0)\}_{i\in\mathbb{N}}$, Lemma \ref{L3.1} does not imply that the phase diameter is strictly decreasing at $t=t_0$. We can construct a solution $\Theta$ in which phase-diameter is nondecreasing in time, even if the framework $(\mathcal{F}1)-(\mathcal{F}2)$ are satisfied.
\vspace{0.3cm}

\begin{lemma}\label{L3.2}
	Suppose there are two increasing sequences $\{i_n\}_{n\in \mathbb{N}}$ and $\{j_n\}_{n\in \mathbb{N}}$ of $\mathbb{N}$ such that 
	\begin{equation}\label{3-2}
		\lim_{n\to\infty}\sum_{k\in\mathbb{N}} \kappa_{i_nk}=0,\quad \lim_{n\to\infty}\sum_{k\in\mathbb{N}} \kappa_{j_nk}=0,\quad \lim_{n\to\infty}\theta_{i_{n}}^{\text{in}}=\sup_{k \in {\mathbb N}}\theta_{k}^{\text{in}},\quad\lim_{n \to\infty}\theta_{j_{n}}^{\text{in}}=\inf_{k \in {\mathbb N}}\theta_{k}^{\text{in}},
	\end{equation}
	and let $\Theta=(\theta_1,\theta_2,\ldots)$ be a solution to \eqref{C-1} with $p=\infty$. Then, the phase-diameter ${\mathcal D}(\Theta)$ is nondecreasing along \eqref{C-1}.
\end{lemma}
\begin{proof}
	From Lemma \ref{L2.2}, one has 
	\[ \sup_{0 \leq t < \infty} |\dot{\theta}_{i}(t)|\leq\sum_{k\in\mathbb{N}} \kappa_{ik},\quad 
	\left| \theta_{i}(t) - \theta_{i}^{\text{in}} \right|\leq t\sum_{k\in\mathbb{N}} \kappa_{ik},\quad \forall~i\in \mathbb{N}.
	\]
	Then, we use the triangle inequality and the above relations to obtain
	\begin{align}
		\begin{aligned} \label{C-1-3-2}
			\left|\theta_{i}\left(t\right)-\theta_{j}\left(t\right)\right| &\geq\left|\theta_{i}^{\text{in}}-\theta_{j}^{\text{in}}\right|-\left|\theta_{i}^{\text{in}}-\theta_{i}\left(t\right)\right|-\left|\theta_{j}^{\text{in}}-\theta_{j}\left(t\right)\right|  \\
			&\geq\left|\theta_{i}^{\text{in}}-\theta_{j}^{\text{in}}\right|-\left(\sum_{k\in\mathbb{N}} \kappa_{ik}+\sum_{k\in\mathbb{N}} \kappa_{jk}\right)t.
		\end{aligned}
	\end{align}
	On the other hand, we use the first two conditions in \eqref{3-2} to  see that  for every $\varepsilon_{1}>0$, there exists a natural number $N=N\left(\varepsilon_{1}\right)\in\mathbb{N}$
	such that 
	\begin{equation} \label{C-1-3-3} 
		n>N \quad \Longrightarrow \quad \sum_{k\in\mathbb{N}} \kappa_{i_nk}<\varepsilon_{1} \quad \mbox{and} \quad  \sum_{k\in\mathbb{N}} \kappa_{j_nk}<\varepsilon_{1}.
	\end{equation}
	For every $\varepsilon_{2}>0$, one can also find $M=M\left(\varepsilon_{2}\right)\in\mathbb{N}$ such that for $n>M$,
	\begin{equation} \label{C-1-3-4}
		\theta_{i_{n}}>\overline{\theta}-\varepsilon_{2} \quad \mbox{and} \quad \theta_{j_{n}}<\underline{\theta}+\varepsilon_{2}. 
	\end{equation}
	Then, by using $\eqref{C-1-3-3}-\eqref{C-1-3-4}$ to the relation \eqref{C-1-3-2} with the index pair $(i_n, j_n)$ with $n \geq N,M$, we have
	\begin{align*}
		{\mathcal D}\left(\Theta\left(t\right)\right) & =\sup_{m,n}\left|\theta_{m}\left(t\right)-\theta_{n}\left(t\right)\right|  \ge\left|\theta_{i_{n}}^{\text{in}}-\theta_{j_{n}}^{\text{in}}\right|-\left(\sum_{k\in\mathbb{N}} \kappa_{i_nk}+\sum_{k\in\mathbb{N}} \kappa_{j_nk}\right)t\\
		& \ge {\mathcal D}(\Theta^{\text{in}} )-2\varepsilon_{1}t-2\varepsilon_{2}.
	\end{align*}
	Since $\varepsilon_{1}$ and $\varepsilon_{2}$ can be arbitrary positive numbers, we can take $\varepsilon_{1},\varepsilon_{2} \to 0$ for each fixed $t$ to obtain the desired result. Therefore, we have
	\[
	{\mathcal D}\left(\Theta\left(t\right)\right) \geq  {\mathcal D}(\Theta^{\text{in}} ), \quad t \geq 0.
	\]
\end{proof}
\vspace{0.3cm}

\begin{remark} \label{R3.2}
	Below, we provide network topology and initial data satisfying a set of relations in $(\mathcal{F}1)-(\mathcal{F}2)$ and \eqref{3-2}.  More precisely, we set 
	\[  \kappa_{ij}=3^{-\left(i+j\right)} \quad \mbox{and} \quad  \theta_{i}^{\text{in}}=\left(-1\right)^{i}\pi/3, \quad i, j \in {\mathbb N}. \]
	Then, one has 
	\[  \sup_{0 \leq t < \infty} \left|\dot{\theta}_{i}(t) \right|\le\sum_{j\in\mathbb{N}} \kappa_{ij}=\frac{1}{2\cdot3^{i-1}} \]
	which yields
	\[
	\left|\theta_{i}^{\text{in}}-\theta_{i}(t) \right|\le\frac{t}{2\cdot3^{i-1}}, \quad t \geq 0. 
	\]
	Therefore, we have
	\[
	\left|\theta_{i}\left(t\right)-\theta_{j}\left(t\right)\right| \ge\left|\theta_{i}^{\text{in}}-\theta_{j}^{\text{in}}\right|-\left|\theta_{i}^{\text{in}}-\theta_{i}\left(t\right)\right|-\left|\theta_{j}^{\text{in}}-\theta_{j}\left(t\right)\right| \ge\left|\theta_{i}^{\text{in}}-\theta_{j}^{\text{in}}\right|-\frac{3t}{2}\left(\frac{1}{3^{i}}+\frac{1}{3^{j}}\right).
	\]
	This gives
	\[
	\mathcal{D}\left(\Theta(t)\right)\ge\left|\theta_{i}^{\text{in}}-\theta_{j}^{\text{in}}\right|-\frac{3t}{2}\left(\frac{1}{3^{i}}+\frac{1}{3^{j}}\right),\quad i,j\in\mathbb{N}.
	\]
	By letting $i=2k+1$ and $j=2k$, we obtain
	\begin{align*}
		\mathcal{D}\left(\Theta(t)\right) & \ge\left|\theta_{2k+1}^{\text{in}}-\theta_{2k}^{\text{in}}\right|-\frac{3t}{2}\left(\frac{1}{3^{2k+1}}+\frac{1}{3^{2k}}\right)=\frac{2\pi}{3}-\frac{2t}{3^{2k}},\quad k\in\mathbb{N}.
	\end{align*}
	Since $\mathcal{D}\left(\Theta^{\text{in}}\right)=\frac{2\pi}{3}$, we have 
	$\mathcal{D}\left(\Theta(t)\right)\ge\frac{2\pi}{3}=\mathcal{D}\left(\Theta^{\text{in}}\right)$.
\end{remark}
\vspace{0.3cm}

Combining the results we have obtained so far, we can characterize the sufficient framework which makes the phase-diameter ${\mathcal D}(\Theta(t))$ constant with respect to $t$. 
\vspace{0.3cm}

\begin{corollary}\label{C3.2} 
	Suppose that network topology and initial data satisfy $(\mathcal{F}1)-(\mathcal{F}2)$ and \eqref{3-2}, and  let $\Theta$ be a solution to \eqref{C-1} with $p=\infty$. Then, the phase-diameter of the configuration $\Theta$ is constant along time:
	\[ {\mathcal D}(\Theta(t)) = {\mathcal D}(\Theta^{\text{in}}), \quad t \geq 0.  \]
\end{corollary}

This counterintuitive example is the case when all particles are moving away from the boundary, but other particles closer to the boundary than we just observed continue to appear with a slower speeds. Macroscopically, it will look as if there is a fixed boundary continuously emitting new particles which velocities are slower for particles emitted later.\\
%The usual technique in the finite Kuramoto model to prove synchronization is comparing two oscillators with maximal and minimal phases.\cite{H-H-K} But this method may not be used for our model, now we have a problem in proving phase synchronization with condition $\sum_{j}\kappa_{ij}\rightarrow0$. From Proposition \ref{P2.1}, one can see that the diameter of the configuration $\Theta(t)$  may not decrease under the condition 
%\begin{equation*}\label{C-2-0}
%	\inf_{i\in\mathbb{N}}\sum_{j\in\mathbb{N}} \kappa_{ij}=0,
%\end{equation*}
%which
This is a unique feature of the countable Kuramoto model compared to the original Kuramoto model with finitely many particles. The sufficient framework leading to the exponential convergence of phase for the homogeneous ensemble as in Proposition \ref{P2.1}(3) will be presented at the end of Section \ref{sec:4}.\\

\subsubsection{Lyapunov functional}
Now, we will analyze the dynamics of \eqref{C-1} with the following symmetric summable network topology, which is the first case of Remark \ref{R2.1}:
\begin{equation}\label{C-2-1}
	K=(\kappa_{ij})\in\ell^{1,1},\quad \kappa_{ij}= \kappa_{ji} \geq  0,\quad i,j\in\mathbb{N}.
\end{equation}
Since $\ell^{1,1}\subset \ell^{p,1}$ for all $p\in [1,\infty]$, one can construct an $\ell^p$-solution by just considering an $\ell^p$ initial data $\Theta^{\text{in}}$ under \eqref{C-2-1}. In addition, under the condition \eqref{C-2-1}, every $\ell^\infty$-solution $\Theta$ to \eqref{C-1} satisfy
\[\dot{\Theta}(t)\in\ell^1, \quad t\geq 0,\]
even when $\Theta$ itself is not contained in $\ell^1$. Note that the condition $\kappa_{ij}=\kappa_{ji}$ also makes the finite Kuramoto model a gradient flow (see Proposition \ref{P2.1}). 
%The book \cite{H-S-D} presents a convergence result for finite-dimensional gradient flow. Specifically, it demonstrates that if a critical point $P$ is an isolated minimum of the potential function, then $P$ is an asymptotically stable equilibrium. We first expect that
%\[\mathcal{A} := \left\{ \Phi = (\theta, \theta, \ldots) \in \ell^{\infty}: ~\inf_{n\in\mathbb{N}}\theta_{n}^{\text{in}}\le\theta\le\sup_{n\in\mathbb{N}}\theta_{n}^{\text{in}}\right\} 
%\]
%is an attractor of $\Theta$, but \[\sum_{i,j\in\mathbb{N}} \kappa_{ij}<\infty \Longrightarrow \inf_{i\in\mathbb{N}} \sum_{j\in\mathbb{N}} \kappa_{ij}=0\] allows a solution with constant diameter. Hence we cannot say $\mathcal{A}$ attracts $\Theta$ in $\ell^\infty$ and $\Theta$ cannot converge to a phase-locked state. But we can prove the asymptotic complete synchronization under the condition \eqref{C-2-1} for network topology in the following theorem.

\begin{theorem}\label{T3.1}
	Suppose that the network topology $(\kappa_{ij})$ satisfies \eqref{C-2-1}, and let $\Theta$ be a solution to \eqref{C-1} with $p=\infty$. Then, we have
	\[ \lim_{t \to \infty} \|\dot{\Theta}(t)\|_2 = 0.  \]
\end{theorem}
\begin{proof} 
	We will apply a Lyapunov functional approach. \newline
	
	\noindent $\bullet$~Step A: First, we suggest the following function as the Lyapunov functional to \eqref{C-1}:
	\begin{equation} \label{C-3}
		P(\Theta) = \frac{1}{2}\sum_{i,j\in\mathbb{N}} \kappa_{ij} (1-\cos(\theta_i - \theta_j)) \geq 0.
	\end{equation}
	Then, we claim:
	\begin{align}
		\begin{aligned}  \label{C-4}
			& \frac{d}{dt}P(\Theta) = -\sum_{i\in\mathbb{N}} |\dot{\theta}_{i}|^{2}=-\|\dot{\Theta}\|_2^2, \\
			& \left|\frac{d}{dt}P(\Theta(t))\right| \leq \sum_{i\in\mathbb{N}} \left(\sum_{j\in\mathbb{N}} \kappa_{ij}\right)^{2}= \|K\|_{2,1}^2\leq \|K\|_{1,1}^2, \\
			&  \left| \frac{d^2}{dt^2}P(\Theta) \right| \leq 2 \|K \|^2_{\infty, 1} \|K\|_{1,1} \leq 2\|K\|_{1,1}^3. 
		\end{aligned}
	\end{align}
	Below, we derive the above estimates in \eqref{C-4} one by one. \newline
	
	\noindent (i)~We differentiate \eqref{C-3} with respect to $t$ and use \eqref{C-1} to find 
	\begin{align}
		\begin{aligned} \label{C-5}
			\frac{d}{dt}P(\Theta) &=\frac{1}{2}\sum_{i,j\in\mathbb{N}} \kappa_{ij}\sin\left(\theta_{i}-\theta_j\right)\left(\dot{\theta}_{i}-\dot{\theta}_{j}\right)\\
			& =\frac{1}{2}\sum_{i,j,k\in\mathbb{N}} \kappa_{ij}\sin\left(\theta_{i}-\theta_j\right)\kappa_{ik}\sin\left(\theta_{k}-\theta_{i}\right)  \\
			&\quad -\frac{1}{2}\sum_{i,j,k\in\mathbb{N}} \kappa_{ij}\sin\left(\theta_{i}-\theta_j\right)\kappa_{jk}\sin\left(\theta_{k}-\theta_j\right)\\
			& \overset{i\leftrightarrow j}{=}-\sum_{i,j,k\in\mathbb{N}}\kappa_{ij}\sin\left(\theta_{i}-\theta_j\right)\kappa_{ik}\sin\left(\theta_{i}-\theta_{k}\right) \\
			&=-\sum_{i\in\mathbb{N}} \left(\sum_{j\in\mathbb{N}}\kappa_{ij}\sin\left(\theta_{i}-\theta_j\right)\right)^{2} =-\sum_{i\in\mathbb{N}}  |\dot{\theta}_{i} |^{2} \leq 0.
		\end{aligned}
	\end{align}
	This also yields
	\[ \Big|  \frac{d}{dt}P(\Theta)  \Big|  = \Bigg|  \sum_{i\in\mathbb{N}} \left(\sum_{j\in\mathbb{N}}\kappa_{ij}\sin\left(\theta_{i}-\theta_j\right)\right)^{2}  \Bigg| \leq \sum_{i\in\mathbb{N}} \left(\sum_{j\in\mathbb{N}}\kappa_{ij} \right)^{2}=\|K\|_{2,1}^2.  
	\]	
	In addition, we differentiate \eqref{C-5} and apply Lemma \ref{L2.2}, Remark \ref{R2.2} and \eqref{C-2-1} to get 
	\[ \Big| \frac{d^2}{dt^2}P(\Theta) \Big| \leq 2  \sum_{i\in\mathbb{N}} |{\dot \theta}_{i} | |{\ddot \theta}_i| \leq { 2} \|K \|_{\infty, 1} \sum_{i\in\mathbb{N}} |{\ddot \theta}_i | \leq 2 \|K \|^2_{\infty, 1} \sum_{i,j\in\mathbb{N}}  \kappa_{ij}  =2\|K\|_{\infty,1}^2\|K\|_{1,1}.  \]
	\noindent $\bullet$~Step B:~Next, we will show that $P$ satisfies all the conditions for Barbalat's lemma, i.e.,  
	\begin{equation*} \label{C-6}
		\exists~\lim_{t \to \infty} P(\Theta(t)), \quad \frac{dP(\Theta)}{dt}~\mbox{is uniformly continuous}.
	\end{equation*}
	The convergence of $P(\Theta)$ comes from the fact that $P(\Theta)$ is a nonincreasing and bounded from below, which we have already verified in { $\eqref{C-4}_1$ and $\eqref{C-4}_2$. In addition, the uniform continuity of $\frac{dP}{dt}$ is a direct consequence of the boundedness of $\frac{d^2P}{dt^2}$, which we have already verified in $\eqref{C-4}_3$.} From \eqref{C-3} and $\eqref{C-4}_1$, $P\left(\Theta(t)\right)$ is a non-increasing function bounded below. Thus, $P(\Theta(t))$ converges as $t \to \infty$. \\
	
	Finally, we apply the differential version of Barbalat's lemma (see Lemma \ref{LA-2}) to conclude
	\[  \lim_{t \to \infty} \frac{dP(\Theta(t))}{dt}  = 0, \quad \mbox{i.e.,} \quad \lim_{t \to \infty} \|\dot{\Theta}(t)\|_2 = 0. \]
\end{proof}

\subsection{$\ell^p$-solution: additional properties for $p<\infty$}\label{sec:3.2}
In this subsection, we will study some special properties for $\ell^p$-solutions which cannot be found from generic $\ell^\infty$-solutions. 
\subsubsection{Strictly decreasing diameter}
If $\Theta(t)\in \ell^p$ for some $1\leq p<\infty$, the only possible limit point for the set $\{\theta_i(t)\}_{i\in\mathbb{N}}$ is $\theta=0$, so that either $\overline{\theta}(t)$ or $\underline{\theta}(t)$ is not a limit point for all time $t$. Then, Lemma \ref{L3.1} implies that under the framework $(\mathcal{F}1)-(\mathcal{F}2)$, the diameter $\mathcal{D}(\Theta(t))$ is strictly decreasing for all $t\geq 0$. However, if there are symmetric matrix $A=(a_{ij})\in \ell^{p,p}$ and $(\kappa_1,\kappa_2,\ldots)\in\ell^{q}$  satisfying \eqref{B-8}:
\begin{equation*}
	\kappa_{ij}=a_{ij}\kappa_j,\quad \forall~i,j\in\mathbb{N}
\end{equation*}
where $q$ is the Holder conjugate of $p$.
Lemma \ref{L2.1} implies that the weighted average 
\[\sum_{i\in \mathbb{N}}\kappa_i\theta_i \]
is a constant of motion of the flow $\Theta\in \mathcal{C}^1(\mathbb{R}_+,\ell^p)$. Therefore, if $\sum_{i\in \mathbb{N}}\kappa_i\theta_i^{\text{in}}$ is nonzero, $\ell^p$-solution $\Theta$ can never converge to a point $(0,0,\ldots)$ in $\ell^p$-norm. However, there exists a possibility for $\Theta$ to converge in $\ell^\infty$-norm (see Section \ref{sec:4}).

\vspace{0.2cm}
\subsubsection{Gradient flow formulation} \label{sec:3.3}

In the finite Kuramoto model, the gradient flow approach plays an essential role in the proof of Proposition \ref{P2.1}. So we wondered if this approach would work as well for infinite-dimensional model. To study the gradient flow structure of model \eqref{C-1}, a suitable space should be equipped with an inner product structure. We found the reason in differential manifold theory.

For a differential Riemannian manifold $({\mathcal M}, g_{ij})$,  let $f: {\mathcal M} \rightarrow\mathbb{R}$ be a smooth function. Then, the gradient vector $\nabla f$ is obtained by identifying differential $df:T_{p}{\mathcal M} \rightarrow T_{p}\mathbb{R}\cong\mathbb{R}$ by a covector $\left\langle \nabla f,-\right\rangle _{\mathbb{R}^{n}}$. Hence in this subsection, we will consider the Cauchy problem for \eqref{C-1} in $\ell^2$-space. We will begin with a Lemma from calculus.

\begin{lemma}\label{L3.4} For $\theta,h\in\mathbb{R}$ with $\left|h\right|<1$,
	one has 
	\[
	{\left|2\sin\left(\theta+\frac{h}{2}\right)\sin\frac{h}{2}-h\sin\theta\right|\le h^{2}.}
	\]
\end{lemma} 
\begin{proof}
	We use elementary inequality: 
	\[
	x-\frac{x^{3}}{6}\leq\sin x\leq x,\qquad\forall~x\ge0,
	\]
	and mean-value theorem to get {
		\begin{align*}
			& \left|2\sin\left(\theta+\frac{h}{2}\right)\sin\frac{h}{2}-h\sin\theta\right|\\
			& \hspace{0.5cm}\le\left|2\sin\left(\theta+\frac{h}{2}\right)\sin\frac{h}{2}-h\sin\left(\theta+\frac{h}{2}\right)\right|+\left|h\sin\left(\theta+\frac{h}{2}\right)-h\sin\theta\right|\\
			& \hspace{0.5cm}\leq2\left|\sin\left(\theta+\frac{h}{2}\right)\right|\Big|\sin\frac{h}{2}-\frac{h}{2}\Big|+|h|\Big|\sin\left(\theta+\frac{h}{2}\right)-\sin\theta\Big|\\
			& \hspace{0.5cm}\le\frac{h^{3}}{24}+\frac{h^{2}}{2}<h^{2}.
		\end{align*}
	}
\end{proof}
\begin{proposition}\label{P3.4} Suppose that the network topology
	$K$ satisfies 
	\[
	K=(\kappa_{ij})\in\ell^{1,1},\quad\kappa_{ij}=\kappa_{ji}>0,\quad i,j\in\mathbb{N}.
	\]
	Then, every $\ell^{2}$-solution to \eqref{C-1} is a gradient flow
	with the potential 
	\[
	P\left(\Theta\right)=\frac{1}{2}\sum_{i,j\in\mathbb{N}}\kappa_{ij}\left(1-\cos\left(\theta_{i}-\theta_{j}\right)\right).
	\]
\end{proposition} 
\begin{proof}
	Let $\boldsymbol{h}=\left\{ h_{k}\right\} _{k\in\mathbb{N}}\in\ell^{2}$
	with $\left\Vert \boldsymbol{h}\right\Vert <\frac{1}{\sqrt{2}}$ and
	\[
	\Phi=\left\{ -\sum_{j\in\mathbb{N}}\kappa_{kj}\sin\left(\theta_{j}-\theta_{k}\right)\right\} _{k\in\mathbb{N}}\in\ell^{2}.
	\]
	From direct calculation, we have 
	\begin{align*}
		\begin{aligned} & P\left(\Theta+\boldsymbol{h}\right)-P\left(\Theta\right)-\left\langle \Phi,\boldsymbol{h}\right\rangle _{2}\\
			& \hspace{1cm}=-\frac{1}{2}\sum_{i,j\in\mathbb{N}}\kappa_{ij}\left(\cos\left(\theta_{i}-\theta_{j}+h_{i}-h_{j}\right)\right)\\
			& \hspace{1cm}+\frac{1}{2}\sum_{i,j\in\mathbb{N}}\kappa_{ij}\left(\cos\left(\theta_{i}-\theta_{j}\right)\right)+\sum_{i,j\in\mathbb{N}}\kappa_{ij}h_{i}\sin\left(\theta_{j}-\theta_{i}\right)\\
			& \hspace{1cm}=-\frac{1}{2}\sum_{i,j\in\mathbb{N}}\kappa_{ij}\left(\cos\left(\theta_{i}-\theta_{j}+h_{i}-h_{j}\right)-\cos\left(\theta_{i}-\theta_{j}\right)\right)+\sum_{i,j\in\mathbb{N}}\kappa_{ij}h_{i}\sin\left(\theta_{j}-\theta_{i}\right)\\
			& \hspace{1cm}\overset{i\leftrightarrow j}{=}-\frac{1}{2}\sum_{i,j\in\mathbb{N}}\kappa_{ij}\left(\cos\left(\theta_{i}-\theta_{j}+h_{i}-h_{j}\right)-\cos\left(\theta_{i}-\theta_{j}\right)+\left(h_{i}-h_{j}\right)\sin\left(\theta_{i}-\theta_{j}\right)\right)\\
			& \hspace{1cm}=-\frac{1}{2}\sum_{i,j\in\mathbb{N}}\kappa_{ij}\left({-2\sin\left(\theta_{i}-\theta_{j}+\frac{h_{i}-h_{j}}{2}\right)\sin\left(\frac{h_{i}-h_{j}}{2}\right)}+\left(h_{i}-h_{j}\right)\sin\left(\theta_{i}-\theta_{j}\right)\right).
		\end{aligned}
	\end{align*}
	Here, Lemma \ref{L3.4} implies 
	\begin{align*}
		{\begin{aligned} & \left|P\left(\Theta+\boldsymbol{h}\right)-P\left(\Theta\right)-\left\langle \Phi,\boldsymbol{h}\right\rangle _{2}\right|\leq\frac{1}{2}\sum_{i,j\in\mathbb{N}}\kappa_{ij}\left|h_{i}-h_{j}\right|^{2},\\
				& \frac{1}{2}\sum_{i,j\in\mathbb{N}}\kappa_{ij}\left|h_{i}-h_{j}\right|^{2}\leq\sum_{i\neq j}\kappa_{ij}\left(h_{i}^{2}+h_{j}^{2}\right)\leq\sum_{i,j\in\mathbb{N}}\kappa_{ij}\left\Vert \boldsymbol{h}\right\Vert _{2}^{2}.
		\end{aligned}}
	\end{align*}
	where we used 
	\[
	\left\Vert \boldsymbol{h}\right\Vert <\frac{1}{\sqrt{2}}\quad\Longrightarrow\quad|h_{i}-h_{j}|\leq1,
	\]
	to meet the assumption in Lemma \ref{L3.4}. Therefore, we conclude
	$dP(\Theta)=\left\langle \Phi,\cdot\right\rangle _{2}$, which is
	equivalent to say that $\Phi$ is the gradient of $P$. 
\end{proof}

Even if we have a gradient structure of \eqref{C-1}, we cannot show the convergence of phases as in Proposition \ref{P2.1}, as the Lojaciewicz inequality in Lemma \ref{LA-4} is  not applicable under the condition \eqref{C-2-1}. To see this, we first assume
\[  \kappa_{ii}=0, \quad i\in\mathbb{N} \]
without loss of generality. Then for each $i \in {\mathbb N}$, we have
\[ \nabla_{\theta_i}P(\Theta) = -\sum_{k\in\mathbb{N}} \kappa_{ik}\sin\left(\theta_{k}-\theta_{i}\right), \]
and the hessian matrix $\nabla^2P(\Theta)=:\left\{ h_{ij}\right\} _{i,j\in\mathbb{N}}$  is given as 
\[ h_{ij}=\frac{\partial}{\partial\theta_j} \Big( \sum_{k\in\mathbb{N}} \kappa_{ik}\sin(\theta_{k}-\theta_{i} ) \Big) = \begin{cases}
	\displaystyle \kappa_{ij}\cos\left(\theta_j-\theta_{i}\right), & i\neq j\\
	\displaystyle -\sum_{k\neq i} \kappa_{ik}\cos\left(\theta_{k}-\theta_{i}\right), & i=j.
\end{cases}  \]
In particular, the hessian matrix at $\Theta=0$ is 
\[  H:=\nabla^2P(\Theta)(0) =  \begin{cases}
	\kappa_{ij} & i\neq j,\\
	-\sum_{k\neq i}\kappa_{ik} & i=j.
\end{cases} \]
Next, we determine the kernel of this Hessian matrix.  Suppose we have two vectors
\[ 
\boldsymbol{v}=\left\{ v_{i}\right\} _{i\in\mathbb{N}}, \quad \boldsymbol{w}=\left\{ w_{i}\right\} _{i\in\mathbb{N}}\in \ell^2.
\]
By direct computation, one has
\[
\left\langle \boldsymbol{w},H\boldsymbol{v}\right\rangle _{l^{2}} =\sum_{i\in\mathbb{N}} w_{i}\left(\sum_{j\in\mathbb{N}} \kappa_{ij}v_{j}-\sum_{j\in\mathbb{N}} \kappa_{ij}v_{i}\right)  \overset{i\leftrightarrow j}{=}-\frac{1}{2}\sum_{i,j\in\mathbb{N}} \kappa_{ij}\left(v_{i}-v_{j}\right)\left(w_{i}-w_{j}\right).
\]
Therefore, we have
\begin{align*}
	\begin{aligned}
		\boldsymbol{v}=\left\{ v_{i}\right\} _{i\in\mathbb{N}}\in\ker H  &\quad \Longleftrightarrow \quad \left \langle \boldsymbol{v},H\boldsymbol{v}\right\rangle _{l^{2}} =-\frac{1}{2}\sum_{i,j\in\mathbb{N}}\kappa_{ij}\left(v_{i}-v_{j}\right)^{2}=0 \\
		& \quad \Longleftrightarrow \quad v_{i}=v_{j} \quad \forall~i,j\in\mathbb{N}, \quad \mbox{i.e.,} \quad \boldsymbol{v} = (v, v, \ldots) \in \ell^2\\
		& \quad \Longleftrightarrow  \quad \boldsymbol{v}=0.
	\end{aligned}
\end{align*}
Now, let $\boldsymbol{e}_{i}$ be a infinite sequence such that all but $i$th element is zero, and the only nonzero element is $1$. Then, we have 
\[
(H\boldsymbol{e}_{i})_j=\begin{cases}
	\kappa_{ij} & j\neq i\\
	-\sum_{k\neq i} \kappa_{ik} & j=i
\end{cases}.
\]
Therefore, the $\ell^2$-norm of $H\boldsymbol{e}_{i}$ can be written as 
\[
\left\Vert H\boldsymbol{e}_{i}\right\Vert _{l^{2}}^2=\sum_{j\in\mathbb{N}} \left(\kappa_{ij}\right)^{2}+\left(\sum_{j\in\mathbb{N}}\kappa_{ij}\right)^{2}.
\]
However, since  $K\in\ell^{1,1}$, we can have
\[ \lim_{i\to\infty}\sum_{j\in\mathbb{N}} \kappa_{ij}=0 \quad \Longrightarrow \quad \lim_{n\to\infty}\left\Vert H\boldsymbol{e}_{n}\right\Vert _{l^{2}}=0, \]
which violates the second condition of \eqref{New-A}.
%%%%%%%%%%%%%%%%%%%%%%%%%%%%%%%%%%%%%%%%%%%%%%%%%%%%%%%%%%%%%%%%%%
%
%   Section 4.
%
%%%%%%%%%%%%%%%%%%%%%%%%%%%%%%%%%%%%%%%%%%%%%%%%%%%%%%%%%%%%%%%%%%

\section{Emergent dynamics of a heterogeneous ensemble} \label{sec:4}
\setcounter{equation}{0}
In this section, we study emergent dynamics of the infinite Kuramoto ensemble which might have several heterogeneous oscillators with nonidentical natural frequencies,
\begin{equation}\label{D-1}
	\begin{cases}
		\displaystyle\dot{\theta}_{i}=\nu_i + \sum_{j\in \mathbb{N}}\kappa_{ij}\sin\left(\theta_{j}-\theta_{i}\right), \quad t>0,\\
		\displaystyle\theta_i(0)=\theta_i^{\text{in}}\in \mathbb{R},\quad \forall~ i\in \mathbb{N},\\
		\displaystyle \Theta^{\text{in}}\in\ell^\infty,\quad \mathcal{V}\in\ell^\infty,\quad K=(\kappa_{ij})\in\ell^{\infty,1}.
	\end{cases}
\end{equation}
In addition to the framework $(\mathcal{F}1)$ -- $(\mathcal{F}2)$, we also consider the following framework in this section:
\vspace{0.3cm}
\begin{itemize}
	\item $(\mathcal{F}3)$: The network topology $K=(\kappa_{ij})$ satisfies 
	\[\inf_{i\in\mathbb{N}}\sum_{j\in\mathbb{N}} \kappa_{ij} =: \| K \|_{-\infty, 1}>0.\]
\end{itemize}

In the sequel, we first prove the existence of a trapping set for infinite heterogeneous ensembles, and then we apply the same argument to prove the convergence of diameter to zero in a homogeneous ensemble as corollary. 

\begin{lemma}\label{L4.1}
	Let $a, b$ and $c$ be positive constants satisfying the relations
	\[ 0\leq c-a\leq \pi-\varepsilon_1, \quad  a-\varepsilon_2\leq b\leq c+\varepsilon_2,\quad 0\leq \varepsilon_{2}\leq \varepsilon_1. \]
	Then, one has 
	\[
	\sin\left(c-a\right)+\sin\left(a-b\right)+\sin\left(b-c\right)\le 4\sin \frac{\varepsilon_{2}}{2}.
	\]
\end{lemma}
\begin{proof} 
	\noindent We use the additive law for trigonometric function to obtain
	\begin{align*}
		& \sin\left(c-a\right)+\sin\left(a-b\right)+\sin\left(b-c\right)\\
		& \hspace{1cm} =\sin\left(c-a\right)+2\sin\left(\frac{a-c}{2}\right)\cos\left(\frac{a-2b+c}{2}\right)\\
		&  \hspace{1cm}  =2\sin\left(\frac{c-a}{2}\right)\left(\cos\left(\frac{a-c}{2}\right)-\cos\left(\frac{a-2b+c}{2}\right)\right)\\
		& \hspace{1cm}  =-4\sin\left(\frac{c-a}{2}\right)\sin\left(\frac{c-b}{2}\right)\sin\left(\frac{b-a}{2}\right)\\
		&\hspace{1cm} =:f(a,b,c).
	\end{align*}
	If $a\leq b\leq c$, then 
	\[ \frac{c-a}{2}, \frac{b-a}{2},\frac{c-b}{2}\in \Big [0,\frac{\pi-\varepsilon_1}{2} \Big ] \]
	and therefore 
	\[ f(a,b,c)\leq 0. \]
	On the other hand, if $a-\varepsilon_2\leq b\leq a$ or $c\leq b\leq c+\varepsilon_2$, then we have
	\[ \frac{c-a}{2}\in [0,\frac{\pi-\varepsilon_1}{2}], \]
	and one of $\frac{b-a}{2}$ and $\frac{c-b}{2}$ is contained in $[0,\frac{\pi-\varepsilon_{1}+\varepsilon_{2}}{2}]$, and the other is contained in $[-\frac{\varepsilon_2}{2},0]$. Therefore, we have 
	\[ f(a,b,c)\leq 4\sin \frac{\varepsilon_{2}}{2}  \quad \mbox{in both cases}. \]
\end{proof}

\begin{lemma}\label{L4.2p}
	Let $\Theta$ be a solution to \eqref{D-1}. For every $t\geq 0$ and $\varepsilon_{2}>0$, consider the following partition of the index set $\mathbb{N}$: 
	\begin{align*}
		\begin{aligned}
			& {\mathcal J}_{1}(\varepsilon_{2},t): =\left\{ k:\theta_{k}\left(t\right)>\overline{\theta}(t)-\varepsilon_{2}\right\}, \\
			& {\mathcal J}_{2}(\varepsilon_{2},t):=\left\{ k:\theta_{k}\left(t\right)<\underline{\theta}(t)+\varepsilon_{2}\right\}, \\
			& {\mathcal J}_{3}(\varepsilon_{2},t):  =\left\{ k:\underline{\theta}(t)+\varepsilon_{2}\le\theta_{k}\left(t\right)\le\overline{\theta}(t)-\varepsilon_{2}\right\}.
		\end{aligned}
	\end{align*}
	Then, if $\mathcal{D}(\Theta(t_0))=\overline{\theta}(t_0)-\underline{\theta}(t_0)= \pi-\varepsilon_1$ for some $\varepsilon_1>0$, we have 
	\begin{equation}\label{D-2}
		\begin{aligned}
			&\dot{\theta}_i(t_0)-\dot{\theta}_j(t_0) \\
			&\hspace{0.2cm} \leq \nu_{i}-\nu_{j} -\sum_{k\in\mathbb{N}}\bigg[\emph{min}(\kappa_{ik},\kappa_{jk})\left[\sin(\theta_i(t_0)-\theta_j(t_0))-4\sin\frac{\varepsilon_{2}}{2} \right] -|\kappa_{ik}-\kappa_{jk}|\sin\varepsilon_{2}\bigg],
		\end{aligned}
	\end{equation}
	for every  $(i,j)\in\mathcal{J}_1(\varepsilon_{2},t_0)\times \mathcal{J}_2(\varepsilon_{2},t_0)$ and sufficiently small $\varepsilon_{2}$ satisfying 
	\begin{equation}\label{D-3p}
		\varepsilon_{2}\leq \varepsilon_{1},\quad \varepsilon_{1}+2\varepsilon_{2}\leq \pi,\quad \sin \varepsilon_{1}\geq 4\sin\frac{\varepsilon_{2}}{2},\quad \sin (\varepsilon_{1}+2\varepsilon_{2})\geq 4\sin\frac{\varepsilon_{2}}{2}.
	\end{equation}
\end{lemma}
\begin{proof}
	One can apply Lemma \ref{L4.1} to 
	\[a=\theta_j(t_0),\quad b=\theta_k(t_0),\quad c=\theta_i(t_0), \]
	for all $i\in\mathcal{J}_1(\varepsilon_{2},t_0), j\in \mathcal{J}_2(\varepsilon_{2},t_0)$ and $k\in\mathbb{N}$ whenever $\varepsilon_{2}\leq \varepsilon_{1}$. Since $\dot{\theta}_i(t_0)-\dot{\theta}_j(t_0)$ can be written as
	\[\dot{\theta}_i(t_0)-\dot{\theta}_j(t_0)=\nu_{i}-\nu_{j}-\sum_{k\in\mathbb{N}} \left(\kappa_{ik}\sin\left(\theta_{i}\left(t_{0}\right)-\theta_{k}\left(t_{0}\right)\right)-\kappa_{jk}\sin\left(\theta_{k}\left(t_{0}\right)-\theta_j\left(t_{0}\right)\right)\right),\]	
	it is sufficient to verify that 
	\[\begin{aligned}
		&\kappa_{ik}\sin\left(\theta_{i}\left(t_{0}\right)-\theta_{k}\left(t_{0}\right)\right)+\kappa_{jk}\sin\left(\theta_{k}\left(t_{0}\right)-\theta_j\left(t_{0}\right)\right)\\
		&\hspace{0.5cm}\geq \mbox{min}(\kappa_{ik},\kappa_{jk})\left[\sin(\theta_i(t_0)-\theta_j(t_0))-4\sin\frac{\varepsilon_{2}}{2} \right]-|\kappa_{ik}-\kappa_{jk}|\sin\varepsilon_{2}
	\end{aligned} \]
	for all $k\in\mathbb{N}$. Note that for sufficiently small $\varepsilon_{2}$ satisfying \eqref{D-3p},  we have 
	\[\begin{aligned}
		&0\leq \mathcal{D}(\Theta(t_0))-2\varepsilon_{2}\leq \theta_i(t_0)-\theta_j(t_0)\leq \mathcal{D}(\Theta(t_0))\leq \pi,\\
		&\sin \mathcal{D}(\Theta(t_0))\geq 4\sin\frac{\varepsilon_{2}}{2},\quad \sin (\mathcal{D}(\Theta(t_0))-2\varepsilon_{2})\geq 4\sin\frac{\varepsilon_{2}}{2}.
	\end{aligned} \]
	These imply
	\[\sin(\theta_i(t_0)-\theta_j(t_0))-4\sin\frac{\varepsilon_{2}}{2}\geq 0 \]
	from the concavity of the sine function on the domain $[0,\pi]$.
	Below, we show the above inequality for $k\in \mathcal{J}_1(\varepsilon_{2},t_0)$, $k\in \mathcal{J}_2(\varepsilon_{2},t_0)$ and $k\in \mathcal{J}_3(\varepsilon_{2},t_0)$ one by one.\\

	\noindent $\bullet$~Case A ($k\in \mathcal{J}_1(\varepsilon_{2},t_0)$): In this case, we use Lemma \ref{L4.1} to get
	\begin{align*}
		&\kappa_{ik}\sin\left(\theta_{i}\left(t_{0}\right)-\theta_{k}\left(t_{0}\right)\right)+\kappa_{jk}\sin\left(\theta_{k}\left(t_{0}\right)-\theta_j\left(t_{0}\right)\right)\\
		&\hspace{0.5cm}= \kappa_{jk}\left[\sin\left(\theta_{i}\left(t_{0}\right)-\theta_{k}\left(t_{0}\right)\right)+\sin\left(\theta_{k}\left(t_{0}\right)-\theta_j\left(t_{0}\right)\right)\right]+(\kappa_{ik}-\kappa_{jk})\sin\left(\theta_{i}\left(t_{0}\right)-\theta_k\left(t_{0}\right)\right)\\
		&\hspace{0.5cm}\geq \kappa_{jk}\left[\sin(\theta_i(t_0)-\theta_j(t_0))-4\sin\frac{\varepsilon_{2}}{2} \right]-|\kappa_{ik}-\kappa_{jk}|\sin\varepsilon_{2}\\
		&\hspace{0.5cm}\geq \mbox{min}(\kappa_{ik},\kappa_{jk})\left[\sin(\theta_i(t_0)-\theta_j(t_0))-4\sin\frac{\varepsilon_{2}}{2} \right]-|\kappa_{ik}-\kappa_{jk}|\sin\varepsilon_{2}.
	\end{align*} 
	\vspace{0.5cm}
	
	\noindent $\bullet$~Case B ($k\in \mathcal{J}_2(\varepsilon_{2},t_0)$): Similar to the first case, we use Lemma \ref{L4.1} to get
	\begin{align*}
		&\kappa_{ik}\sin\left(\theta_{i}\left(t_{0}\right)-\theta_{k}\left(t_{0}\right)\right)+\kappa_{jk}\sin\left(\theta_{k}\left(t_{0}\right)-\theta_j\left(t_{0}\right)\right)\\
		&\hspace{0.5cm}= \kappa_{ik}\left[\sin\left(\theta_{i}\left(t_{0}\right)-\theta_{k}\left(t_{0}\right)\right)+\sin\left(\theta_{k}\left(t_{0}\right)-\theta_j\left(t_{0}\right)\right)\right]+(\kappa_{jk}-\kappa_{ik})\sin\left(\theta_{k}\left(t_{0}\right)-\theta_j\left(t_{0}\right)\right)\\
		&\hspace{0.5cm}\geq \kappa_{ik}\left[\sin(\theta_i(t_0)-\theta_j(t_0))-4\sin\frac{\varepsilon_{2}}{2} \right]-|\kappa_{ik}-\kappa_{jk}|\sin\varepsilon_{2}\\
		&\hspace{0.5cm}\geq \mbox{min}(\kappa_{ik},\kappa_{jk})\left[\sin(\theta_i(t_0)-\theta_j(t_0))-4\sin\frac{\varepsilon_{2}}{2} \right]-|\kappa_{ik}-\kappa_{jk}|\sin\varepsilon_{2}.
	\end{align*}
	\vspace{0.5cm}
	
	\noindent $\bullet$~Case C ($k\in \mathcal{J}_3(\varepsilon_{2},t_0)$): Again, in this case, we use Lemma \ref{L4.1} to get
	\begin{align*}
		&\kappa_{ik}\sin\left(\theta_{i}\left(t_{0}\right)-\theta_{k}\left(t_{0}\right)\right)+\kappa_{jk}\sin\left(\theta_{k}\left(t_{0}\right)-\theta_j\left(t_{0}\right)\right)\\
		&\hspace{0.5cm}\geq \mbox{min}(\kappa_{ik},\kappa_{jk})\left[\sin\left(\theta_{i}\left(t_{0}\right)-\theta_{k}\left(t_{0}\right)\right)+\sin\left(\theta_{k}\left(t_{0}\right)-\theta_j\left(t_{0}\right)\right)\right]\\
		&\hspace{0.5cm}\geq \mbox{min}(\kappa_{ik},\kappa_{jk})\sin(\theta_i(t_0)-\theta_j(t_0))\\
		&\hspace{0.5cm}\geq \mbox{min}(\kappa_{ik},\kappa_{jk})\left[\sin(\theta_i(t_0)-\theta_j(t_0))-4\sin\frac{\varepsilon_{2}}{2} \right]-|\kappa_{ik}-\kappa_{jk}|\sin\varepsilon_{2}.
	\end{align*}
	Finally, we combine estimates in Case A $\sim$ Case C to obtain
	\begin{align*}
		\dot{\theta}_i(t_0)-\dot{\theta}_j(t_0)&=\nu_{i}-\nu_{j}-\sum_{k\in\mathbb{N}} \left(\kappa_{ik}\sin\left(\theta_{i}\left(t_{0}\right)-\theta_{k}\left(t_{0}\right)\right)-\kappa_{jk}\sin\left(\theta_{k}\left(t_{0}\right)-\theta_j\left(t_{0}\right)\right)\right)\\
		&\leq \nu_{i}-\nu_{j}-\sum_{k\in\mathbb{N}}\bigg[\mbox{min}(\kappa_{ik},\kappa_{jk})\left[\sin(\theta_i(t_0)-\theta_j(t_0))-4\sin\frac{\varepsilon_{2}}{2} \right]\\
		&\hspace{1.65cm}-|\kappa_{ik}-\kappa_{jk}|\sin\varepsilon_{2}\bigg]
	\end{align*}
	for every $\varepsilon_{2}\leq \varepsilon_{1}$ and $(i,j)\in\mathcal{J}_1(\varepsilon_{2},t_0)\times \mathcal{J}_2(\varepsilon_{2},t_0)$,	which is our desired result.
\end{proof}

If $\Theta$ is a solution to \eqref{D-1} under the framework $(\mathcal{F}1)-(\mathcal{F}3)$, one can further estimate the right-hand side of \eqref{D-2} to obtain the following result, so-called `practical synchronization'.
\begin{theorem}\label{T4.1}
	Assume that the initial data $\Theta^{\text{in}}$ and network topology $(\kappa_{ij})$ satisfy $(\mathcal{F}1)$ -- $(\mathcal{F}3)$. Assume further that
	\[0<\mathcal{D}(\mathcal{V})<\|\boldsymbol{\tilde \kappa}\|_1\|K\|_{-\infty,1},\quad \mathcal{D}(\Theta^{\text{in}})\in (\gamma,\pi-\gamma),\quad \gamma=\sin^{-1}\left(\frac{\mathcal{D}(\mathcal{V})}{\|\boldsymbol{\tilde \kappa}\|_1\|K\|_{-\infty,1}} \right)<\frac{\pi}{2}. \]
	Then, if $\Theta$ is an $\ell^\infty$-solution to \eqref{D-1}, the phase diameter $\mathcal{D}(\Theta(t))$ has the following asymptotic upper bound:
	\[\limsup_{t\to\infty} \mathcal{D}(\Theta(t))\leq \gamma. \]
\end{theorem}
\begin{proof}
	Fix a sufficiently small $\varepsilon_0>0$ satisfying 
	\[\sin^{-1}\left(\frac{\mathcal{D}(\mathcal{V})+\|K\|_{\infty,1}\varepsilon_0}{\|\boldsymbol{\tilde \kappa}\|_1\|K\|_{-\infty,1}} \right)+\frac{\varepsilon_0}{3} \leq\mathcal{D}(\Theta^{\text{in}})\leq \pi-\sin^{-1}\left(\frac{\mathcal{D}(\mathcal{V})+\|K\|_{\infty,1}\varepsilon_0}{\|\boldsymbol{\tilde \kappa}\|_1\|K\|_{-\infty,1}} \right),  \]
	and define a positive number $\varepsilon_{2}$ as
	\[\varepsilon_{2}=\mbox{min}\left(\frac{\sin\gamma}{2},\frac{\varepsilon_0}{6}\right). \]
	Then, we will show the following statement in the proof.\\
	
	{\it	We claim: whenever the phase diameter $\mathcal{D}(\Theta)$ at time $t=t_0$ satisfy
		\[\sin^{-1}\left(\frac{\mathcal{D}(\mathcal{V})+\|K\|_{\infty,1}\varepsilon_0}{\|\boldsymbol{\tilde \kappa}\|_1\|K\|_{-\infty,1}} \right)+\frac{\varepsilon_0}{3}\leq\mathcal{D}(\Theta(t_0))\leq \pi-\sin^{-1}\left(\frac{\mathcal{D}(\mathcal{V})+\|K\|_{\infty,1}\varepsilon_0}{\|\boldsymbol{\tilde \kappa}\|_1\|K\|_{-\infty,1}} \right),\]
		we have 
		\[\mathcal{D}(\Theta(t))\leq\mathcal{D}(\Theta(t_0))-\|K\|_{\infty,1}\varepsilon_{2}(t-t_0),\quad \forall~ t_0\leq t\leq t_0+\frac{\varepsilon_{2}}{2(\mathcal{D}(\mathcal{V})+2\|K\|_{\infty,1})}.\]}
	
	To see this, we first verify that $\varepsilon_{2}$ satisfies the condition \eqref{D-3p} for given $\varepsilon_{1}=\pi-\mathcal{D}(\Theta(t_0))$. If $\varepsilon_{2}$ is a positive number satisfying 
	\begin{equation}\label{D-5}
		{ 4\varepsilon_{2}\leq \sin\varepsilon_{1} \leq \pi-\varepsilon_{1},}
	\end{equation}
	one can easily see that $\varepsilon_{2}$ satisfies \eqref{D-3p}. However, \eqref{D-5} follows from the fact that 
	\[\begin{aligned}
		&2\varepsilon_{2}\leq \sin\gamma\leq \sin\mathcal{D}(\Theta(t_0))=\sin\varepsilon_{1},\\
		&2\varepsilon_{2}\leq \frac{\varepsilon_0}{3}\leq \sin^{-1}\left(\frac{\mathcal{D}(\mathcal{V})+\|K\|_{\infty,1}\varepsilon_0}{\|\boldsymbol{\tilde \kappa}\|_1\|K\|_{-\infty,1}} \right)+\frac{\varepsilon_0}{3}\leq\mathcal{D}(\Theta(t_0))=\pi-\varepsilon_{1}.
	\end{aligned} \]
	Now, consider any index pair $(i,j)\in\mathcal{J}_1(\varepsilon_{2},t_0)\times \mathcal{J}_2(\varepsilon_{2},t_0)$. Then, Lemma \ref{L4.2p} yields
	\begin{align*}
		\dot{\theta}_i(t_0)-\dot{\theta}_j(t_0)&\leq \nu_{i}-\nu_{j}-\sum_{k\in\mathbb{N}}\bigg[\mbox{min}(\kappa_{ik},\kappa_{jk})\left[\sin(\theta_i(t_0)-\theta_j(t_0))-4\sin\frac{\varepsilon_{2}}{2} \right]\\
		&\hspace{1.65cm}-|\kappa_{ik}-\kappa_{jk}|\sin\varepsilon_{2}\bigg]\\
		&\leq \mathcal{D}(\mathcal{V})-\left(\sum_{k\in\mathbb{N}}\tilde{\kappa}_k\right)\mbox{min}\left(\sum_{l\in\mathbb{N}}\kappa_{il},\sum_{l\in\mathbb{N}}\kappa_{jl} \right)\left[\sin(\theta_i(t_0)-\theta_j(t_0))-4\sin\frac{\varepsilon_{2}}{2} \right]\\
		&\hspace{1.65cm}+\sum_{k\in\mathbb{N}}|\kappa_{ik}-\kappa_{jk}|\sin\varepsilon_{2}\\
		&\leq \mathcal{D}(\mathcal{V})-\| {\boldsymbol{\tilde{\kappa}}} \|_1\|K\|_{-\infty,1}\left[\sin(\theta_i(t_0)-\theta_j(t_0))-4\sin\frac{\varepsilon_{2}}{2} \right]+2\|K\|_{\infty,1}\sin\varepsilon_{2}\\
		&\leq (\mathcal{D}(\mathcal{V})+8\|K\|_{\infty,1}\sin\frac{\varepsilon_{2}}{2})-\| {\boldsymbol{\tilde{\kappa}}} \|_1\|K\|_{-\infty,1}\sin(\theta_i(t_0)-\theta_j(t_0)),
	\end{align*}
	where we used
	\[\| {\boldsymbol{\tilde{\kappa}}} \|_1\leq 1,\quad \|K\|_{-\infty,1}\leq \|K\|_{\infty,1} \]
	in the last inequality. Then, by using Lemma \ref{L2.3}, we have 
	\begin{align*}
		\dot{\theta}_i(t)-\dot{\theta}_j(t)&\leq (\mathcal{D}(\mathcal{V})+8\|K\|_{\infty,1}\sin\frac{\varepsilon_{2}}{2})-\| {\boldsymbol{\tilde{\kappa}}} \|_1\|K\|_{-\infty,1}\sin(\theta_i(t_0)-\theta_j(t_0))\\
		&\hspace{0.5cm}+2(t-t_0)\|K\|_{\infty,1}(\mathcal{D}(\mathcal{V})+2\|K\|_{\infty,1} )
	\end{align*}
	for all $t\in\mathbb{R}_+$. In particular, we have 
	\begin{align*}
		\dot{\theta}_i(t)-\dot{\theta}_j(t)&\leq (\mathcal{D}(\mathcal{V})+5\|K\|_{\infty,1}\varepsilon_{2})-\| {\boldsymbol{\tilde{\kappa}}} \|_1\|K\|_{-\infty,1}\sin(\theta_i(t_0)-\theta_j(t_0))\\
		&\leq 5\|K\|_{\infty,1}\varepsilon_{2}-\|K\|_{\infty,1}\varepsilon_0\\
		&\leq -\|K\|_{\infty,1}\varepsilon_{2},
	\end{align*}
	for all $|t-t_0|\leq \frac{\varepsilon_{2}}{2(\mathcal{D}(\mathcal{V})+2\|K\|_{\infty,1})}$.
	On the other hand, if $(i',j')$ is not contained in $\mathcal{J}_1(\varepsilon_{2},t_0)\times \mathcal{J}_2(\varepsilon_{2},t_0)$ or $\mathcal{J}_2(\varepsilon_{2},t_0)\times \mathcal{J}_1(\varepsilon_{2},t_0)$, we have 
	\[|\theta_{i'}(t_0)-\theta_{j'}(t_0)|\leq \mathcal{D}(\Theta(t_0))-\varepsilon_{2}. \]
	Therefore, by using Lemma \ref{L2.3}, we have
	\[\begin{aligned}
		|\theta_{i'}(t)-\theta_{j'}(t)|&\leq |\theta_{i'}(t_0)-\theta_{j'}(t_0)|+(\mathcal{D}(\mathcal{V})+2\|K\|_{\infty,1})|t-t_0|\\
		&\leq \mathcal{D}(\Theta(t_0))-\frac{\varepsilon_{2}}{2},
	\end{aligned} \]
	for all $|t-t_0|\leq \frac{\varepsilon_{2}}{2(\mathcal{D}(\mathcal{V})+2\|K\|_{\infty,1})}$.
	To sum up, the phase diameter $\mathcal{D}(\Theta)$ satisfies
	\begin{equation}\label{D-6}
		\begin{aligned}
			\mathcal{D}(\Theta(t))&\leq \mbox{max}\left(\mathcal{D}(\Theta(t_0))-\|K\|_{\infty,1}\varepsilon_{2}(t-t_0), \mathcal{D}(\Theta(t_0))-\frac{\varepsilon_{2}}{2}\right)\\
			&=\mathcal{D}(\Theta(t_0))-\|K\|_{\infty,1}\varepsilon_{2}(t-t_0),
		\end{aligned} 
	\end{equation}
	for all $t\in [t_0,t_0+\frac{\varepsilon_{2}}{2(\mathcal{D}(\mathcal{V})+2\|K\|_{\infty,1})}]$, where we used 
	\[\|K\|_{\infty,1}\varepsilon_{2}\cdot\frac{\varepsilon_{2}}{2(\mathcal{D}(\mathcal{V})+2\|K\|_{\infty,1})}\leq \frac{\varepsilon_{2}^2}{4}<\frac{\varepsilon_{2}}{2} \]
	to find a bigger one in \eqref{D-6}.\\
	
	Once we prove the aforementioned claim, we can conclude that $\mathcal{D}(\Theta(t))$ reaches the lower bound 
	\[\sin^{-1}\left(\frac{\mathcal{D}(\mathcal{V})+\|K\|_{\infty,1}\varepsilon_0}{\|\boldsymbol{\tilde \kappa}\|_1\|K\|_{-\infty,1}} \right)+\frac{\varepsilon_0}{3}\]
	in a finite time. Since $\varepsilon_0$ can be chosen arbitrary small, we have 
	\[\limsup_{t\to\infty} \mathcal{D}(\Theta(t))\leq \inf_{\varepsilon_0>0}\left[\sin^{-1}\left(\frac{\mathcal{D}(\mathcal{V})+\|K\|_{\infty,1}\varepsilon_0}{\|\boldsymbol{\tilde \kappa}\|_1\|K\|_{-\infty,1}} \right)+\frac{\varepsilon_0}{3}\right]=\gamma, \]
	which is our desired result.
\end{proof}

\begin{remark} Note that the quantity $\|K \|_{-\infty, 1}$ measures the degree of coupling strengths. Therefore, Theorem \ref{T4.1} yields that the phase-diameter can be made as small as we want by increasing the quantity $\| K \|_{-\infty, 1}$. This is in fact called ''practical synchronization" as discussed in \cite{H-N-P2, K-Y-K-S, M-Z-C1, M-Z-C2} for synchronization models with finite system size. 
\end{remark}

%\subsection{Application of local bound estimate} \label{sec:4.2}
%In this subsection, we provide two corollaries of Theorem \ref{T4.1}.  We return to the homogeneous ensemble. 
\begin{corollary}\label{C4.1}
	Suppose that the initial data $\Theta^{\text{in}}$ and network topology $(\kappa_{ij})$ satisfy $(\mathcal{F}1)-(\mathcal{F}3)$, and we assume that all natural frequencies are identical:
	\[ \mathcal{D}(\mathcal{V})=0. \]
	If $\Theta$ is an $\ell^\infty$-solution to \eqref{D-1}, the phase diameter $\mathcal{D}(\Theta(t))$ converges to zero exponentially. 
\end{corollary}
\begin{proof}
	In this case, we do not fix a constant $\varepsilon_0$, but instead define 
	\[\varepsilon_{2}:=\frac{\| {\boldsymbol{\tilde{\kappa}}} \|_1\|K\|_{-\infty,1}}{(10\|K\|_{\infty,1}+4\| {\boldsymbol{\tilde{\kappa}}} \|_1\|K\|_{-\infty,1})}\sin\mathcal{D}(\Theta(t_0))\] 
	and verify \eqref{D-5} immediately:
	\[2\varepsilon_2\leq \sin\mathcal{D}(\Theta(t_0))=\sin(\pi-\varepsilon_{1})=\sin\varepsilon_{1},\quad \sin(\pi-\varepsilon_{1})\leq \pi-\varepsilon_{1}. \]
	Then, for every $(i,j)\in\mathcal{J}_1(\varepsilon_{2},t_0)\times\mathcal{J}_2(\varepsilon_{2},t_0)$, one can apply Lemma \ref{L4.2p} as in Theorem \ref{T4.1} to obtain
	\begin{align*}
		\dot{\theta}_i(t)-\dot{\theta}_j(t)&\leq 5\|K\|_{\infty,1}\varepsilon_{2}-\| {\boldsymbol{\tilde{\kappa}}} \|_1\|K\|_{-\infty,1}\sin(\theta_i(t_0)-\theta_j(t_0))\\
		&\leq 5\|K\|_{\infty,1}\varepsilon_{2}-\| {\boldsymbol{\tilde{\kappa}}} \|_1\|K\|_{-\infty,1}\mbox{min}\left(\sin(\mathcal{D}(\Theta(t_0))-2\varepsilon_2), \sin\mathcal{D}(\Theta(t_0))\right)\\
		&\leq 5\|K\|_{\infty,1}\varepsilon_{2}-\| {\boldsymbol{\tilde{\kappa}}} \|_1\|K\|_{-\infty,1}(\sin(\mathcal{D}(\Theta(t_0)))-2\varepsilon_{2})\\
		&=(5\|K\|_{\infty,1}+2\| {\boldsymbol{\tilde{\kappa}}} \|_1\|K\|_{-\infty,1})\varepsilon_{2}-\| {\boldsymbol{\tilde{\kappa}}} \|_1\|K\|_{-\infty,1}\sin(\mathcal{D}(\Theta(t_0)))\\
		&=-\frac{1}{2}\| {\boldsymbol{\tilde{\kappa}}} \|_1\|K\|_{-\infty,1}\sin(\mathcal{D}(\Theta(t_0))),
	\end{align*}
	for all $|t-t_0|\leq\frac{\varepsilon_{2}}{4\|K\|_{\infty,1}}$.\\
	
	Now, define $t_0:=0$ and set $t_1,t_2,\ldots$ iteratively by using the relation
	\[t_{n+1}=t_n+\frac{1}{56\|K\|_{\infty,1}}\sin\mathcal{D}(\Theta(t_n)). \]
	Then, we have the following series of inequalities:
	\[\begin{aligned}
		\mathcal{D}(\Theta(t_{n+1}))\leq \mathcal{D}(\Theta(t_{n}))-\frac{1}{112}\sin^2\mathcal{D}(\Theta(t_{n}))\leq \mathcal{D}(\Theta(t_{n}))-\frac{\sin^2\mathcal{D}(\Theta^{\text{in}})}{112\mathcal{D}(\Theta^{\text{in}})^2}\mathcal{D}(\Theta(t_{n}))^2.
	\end{aligned} \]
	Therefore, we have 
	\[\begin{aligned}
		\mathcal{D}(\Theta(t_n))\leq \frac{1}{\frac{1}{\mathcal{D}(\Theta^{\text{in}})}+n\cdot \frac{112\mathcal{D}(\Theta^{\text{in}})^2}{\sin^2\mathcal{D}(\Theta^{\text{in}})}},\quad t_n\lesssim \frac{\sin^2\mathcal{D}(\Theta^{\text{in}})}{112\mathcal{D}(\Theta^{\text{in}})^2}\cdot \frac{1}{56\|K\|_{\infty,1}}\cdot \log n,
	\end{aligned} \]
	which shows the exponential convergence of $\mathcal{D}(\Theta(t))$ with respect to $t$.

\end{proof}

\section{Kuramoto ensemble on a sender network} \label{sec:5}
\setcounter{equation}{0}
In this section, we consider a network topology in which capacity at the $i$-th node depends only on neighboring nodes: 
\begin{equation} \label{E-0}
	\kappa_{ij} = \kappa_j > 0, \quad  i, j \in {\mathbb N} \quad \mbox{and} \quad  \sum_{i\in\mathbb{N}} \kappa_i = 1,
\end{equation}
which represents the second case of Remark \ref{R2.1}. For this network topology, we can derive synchronization estimates for an infinite homogeneous Kuramoto ensemble without any restriction on the size of the initial configuration. Furthermore, we can also derive an exponential synchronization for a heterogeneous ensemble.\\

Consider the Cauchy problem for an infinite Kuramoto model over sender network \eqref{E-0}:
\begin{equation}\label{E-1}
	\begin{cases}
		\displaystyle \dot{\theta}_{i}= \nu_i + \sum_{j\in\mathbb{N}} \kappa_j\sin\left(\theta_j-\theta_{i}\right), \quad t>0, \\
		\displaystyle \theta_i(0)=\theta_i^{\text{in}},\quad \forall~ i\in \mathbb{N}.
	\end{cases}
\end{equation}
In the following two subsections, we study the emergent dynamics of \eqref{E-1} for homogeneous and heterogeneous ensembles, respectively. 
\subsection{Homogeneous ensemble} \label{sec:5.1}
In this subsection, we consider the homogeneous ensemble with the same natural frequencies:
\[ \nu_i = \nu, \quad i \in {\mathbb N}. \]
Then, as discussed in Section \ref{sec:2}, we may assume that $\nu = 0$ without loss of generality. 
\begin{equation}\label{E-1-1}
	\begin{cases}
		\displaystyle \dot{\theta}_{i}=  \sum_{j\in\mathbb{N}} \kappa_j\sin\left(\theta_j-\theta_{i}\right), \quad t>0, \\
		\displaystyle \theta_i(0)=\theta_i^{\text{in}},\quad \forall~ i\in \mathbb{N}.
	\end{cases}
\end{equation}

\subsubsection{Order parameters} \label{sec:5.1.1} This part introduces the order parameters associated with \eqref{E-1-1}. A polar representation of the weighted sum of $z_i$ gives the order parameters $(r, \phi)$ for the phase configuration $\Theta$:
\begin{equation} \label{E-2}
	re^{{\rm{i}}\phi} :=\sum_{k\in\mathbb{N}} \kappa_{k}e^{{\rm{i}}\theta_{k}}.
\end{equation}
This is equivalent to
\begin{equation}\label{E-3} 
	re^{{\rm{i}}\left(\phi-\theta_{i}\right)}=\sum_{k\in\mathbb{N}} \kappa_{k}e^{{\rm{i}}\left(\theta_{k}-\theta_{i}\right)}.
\end{equation}
We now compare the imaginary part of \eqref{E-3} to obtain
\[ r\sin\left(\phi-\theta_{i}\right)=\sum_{k\in\mathbb{N}} \kappa_{k}\sin\left(\theta_{k}-\theta_{i}\right). \]
Then, we use the above relation and $\eqref{E-1-1}_1$ to rewrite infinite Kuramoto model using order parameters:
\begin{equation} \label{E-4}
	\dot{\theta}_{i} = r\sin(\phi - \theta_i).
\end{equation}

In the following lemma, we study the governing system for $(r, \phi)$.
\begin{lemma} \label{L5.1}
	Let $\Theta$ be a solution to \eqref{E-1-1}. Then, the order parameters $(r, \phi)$ satisfy
	\[ 
	\begin{cases}
		\displaystyle \dot{r} =r\sum_{k\in\mathbb{N}} \kappa_{k}\sin^2\left(\theta_{k}-\phi\right),\quad t>0, \\
		\displaystyle \dot{\phi} =\sum_{k\in\mathbb{N}} \kappa_{k}\cos\left(\theta_{k}-\phi\right)\sin\left(\theta_{k}-\phi\right).
	\end{cases}
	\]
\end{lemma}
\begin{proof} We differentiate \eqref{E-2} to find
	\[ \dot{r}e^{{\rm{i}}\phi}+{\rm{i}}re^{{\rm{i}}\phi}\dot{\phi}={\rm{i}}\sum_{k\in\mathbb{N}} \kappa_{k}e^{{\rm{i}}\theta_{k}}\cdot\dot{\theta}_{k}. \]
	Now, we divide the above relation by $e^{{\mathrm i}\phi}$ to see
	\begin{align*}
		\begin{aligned}
			& \dot{r}  +{\rm{i}} r  \dot{\phi}={\rm{i}}\sum_{k\in\mathbb{N}} \kappa_{k}e^{{\rm{i}} (\theta_{k} - \phi)}\cdot\dot{\theta}_{k}  = -\sum_{k\in\mathbb{N}}  \kappa_k {\dot \theta}_k \sin(\theta_k - \phi) + {\mathrm i} \sum_{k\in\mathbb{N}} \kappa_k {\dot \theta}_k \cos(\theta_k - \phi). 
		\end{aligned}
	\end{align*}
	We compare the real and imaginary parts of the above relation and use  \eqref{E-4}. More precisely, \newline
	
	\noindent $\bullet$~(Real part):~From direct calculation, we have 
	\[
	\dot{r} = \sum_{k\in\mathbb{N}}  \kappa_k {\dot \theta}_k \sin(\phi- \theta_k) = r \sum_{k\in\mathbb{N}}  \kappa_k \sin^2(\phi - \theta_k).
	\]
	\noindent $\bullet$~(Imaginary part):  Similarly, one has 
	\[
	r  \dot{\phi} = \sum_{k\in\mathbb{N}} \kappa_k {\dot \theta}_k \cos(\theta_k - \phi) = r \sum_{k\in\mathbb{N}} \kappa_k \sin(\phi - \theta_k)\cos(\phi - \theta_k).
	\]
	Therefore, we have the desired dynamics for $\phi$ as long as $r > 0$. 
\end{proof}
Next, we study the asymptotic behaviors of \eqref{E-1-1}. 
\begin{proposition}\label{P5.1}
	Let $\Theta$ be a solution to \eqref{E-1-1}. Then, the following dichotomy holds:
	\[ \emph{Either}~r(t)\equiv  0 \quad \emph{or} \quad \lim_{t \to \infty}  \sum_{k\in\mathbb{N}} \kappa_{k}\sin^{2}\left(\theta_{k}(t) -\phi(t) \right) = 0. \]
\end{proposition}
\begin{proof}
	Below, we consider two cases:
	\[ r(0) = 0, \quad r(0) > 0. \]
	\noindent $\bullet$~Case A $(r(0) = 0)$: In this case, we employ the uniqueness of ODEs to \eqref{E-4} and obtain
	\[ \theta_i(t) = \theta_i^{\text{in}}, \quad i \in {\mathbb N}, \quad \mbox{i.e.,} \quad r(t) = r(0) = 0. \]

	\noindent $\bullet$~Case B $(r(0) > 0)$: In this case, $r$ is uniformly bounded by one and monotonically increasing in time $t$ (Lemma \ref{L5.1}) :
	\[ r(t) \leq \sum_{k\in\mathbb{N}} \kappa_k = 1, \quad {\dot r}(t) \geq  0, \quad t > 0. \]
	Therefore, there exists a positive real number $r^{\infty} \in [r(0), 1]$ such that 
	\[\lim_{t \to \infty} r(t) = r^{\infty}. \]
	Now, we claim that 
	\begin{equation} \label{E-6}
		\int_0^{\infty} \sum_{k\in\mathbb{N}} \kappa_{k}\sin^{2}\left(\theta_{k}(t)-\phi(t) \right)dt < \infty \quad \mbox{and} \quad  \left|\frac{d}{dt}\left(\sum_{k} \kappa_{k}\sin^{2}\left(\theta_{k}-\phi\right)\right)\right| \leq 4.
	\end{equation}
	
	\vspace{0.2cm}
	
	\noindent $\diamond$~(Derivation of the first relation in \eqref{E-6}): By using Lemma \ref{L5.1}, we have
	\[
	\frac{\dot{r}}{r}=\sum_{k} \kappa_{k}\sin^{2}\left(\theta_{k}-\phi\right) \quad  \Longrightarrow \quad \ln\left(r\left(t\right)\right)-\ln\left(r\left(0\right)\right)=\int_{0}^{t}\sum_{k\in\mathbb{N}} \kappa_{k}\sin^{2}\left(\theta_{k}\left(s\right)-\phi\left(s\right)\right)ds.
	\]
	Therefore, we take a limit $t \to \infty$ to obtain $\eqref{E-6}_1$. \newline
	
	\noindent $\diamond$~(Derivation of the second relation in \eqref{E-6}): By direct calculation, we have
	\begin{align*}
		& \left|\frac{d}{dt}\left(\sum_{k\in\mathbb{N}}  \kappa_{k}\sin^{2}\left(\theta_{k}-\phi\right)\right)\right|\\
		&\hspace{0.5cm}  =\left|2\sum_{k\in\mathbb{N}} \kappa_{k}\left(\dot{\theta}_{k}-\dot{\phi}\right)\sin\left(\theta_{k}-\phi\right)\cos\left(\theta_{k}-\phi\right)\right| \le2\sum_{k\in\mathbb{N}} \kappa_{k}\left(\left|\dot{\theta}_{k}\right|+\left|\dot{\phi}\right|\right)\\
		&\hspace{0.5cm} = 2\sum_{k\in\mathbb{N}} \kappa_{k} \left(\left|\sum_{m=1}^{\infty} \kappa_m \sin (\theta_m - \theta_k)\right|+\left| \sum_{m=1}^{\infty} \kappa_{m}\cos\left(\theta_{m}-\phi\right)\sin\left(\theta_{m}-\phi\right) \right|\right)\\
		&\hspace{0.5cm}  \le 4 \Big( \sum_{k\in\mathbb{N}} \kappa_{k}  \Big)^2  =4.
	\end{align*}
	Finally, we apply the integral version of Barbalat\rq{}s lemma for $\sum_{k\in\mathbb{N}} \kappa_{k}\sin^{2}\left(\theta_{k}(t)-\phi(t) \right)$ to get the zero convergence. 
\end{proof}
As a direct application of Proposition \ref{P5.1}, we have the complete synchronization of \eqref{E-1-1}. 
\begin{theorem} \label{T5.1}
	Let $\Theta$ be a solution to \eqref{E-1-1}. Then,  the following assertions hold. 
	\begin{enumerate}
		\item
		Complete synchronization emerges asymptotically:
		\[  \lim_{t \to \infty} |{\dot \theta}_i(t) - {\dot \theta}_j(t)| = 0, \quad i, j \in {\mathbb N}. \]
		\item
		If $r(0) > 0$, then for each pair $(i, j)$, there exists an integer $n_{ij}$ such that 
		\[ \lim_{t \to \infty} (\theta_i(t) - \theta_j(t)) = n_{ij} \pi. \]
	\end{enumerate}
\end{theorem}
\begin{proof}
	(i)~For the case in which $r(0) = 0$, we have
	\[ {\dot \theta}_i(t) = 0, \quad t > 0, \]
	which is indeed a steady-state solution. Now, we consider a generic case in which $r(0) > 0$. In this case, the dichotomy in Proposition \ref{P5.1} yields
	\begin{equation} \label{E-7}
		\lim_{t \to \infty}  \sum_{k\in\mathbb{N}} \kappa_{k}\sin^{2}\left(\theta_{k}(t) -\phi(t) \right) = 0. 
	\end{equation}
	On the other hand, it follows from Lemma \ref{L5.1} and \eqref{E-4} that
	\begin{equation} \label{E-8}
		\sin(\phi - \theta_i) = \frac{{\dot \theta}_i}{r} .  
	\end{equation} 
	Finally, we combine \eqref{E-7} and \eqref{E-8} to obtain
	\[
	{\lim_{t \to \infty}  \frac{\sum_{k\in\mathbb{N}} \kappa_{k} |{\dot \theta}_k(t)|^2}{r^2(t)} = 0.}
	\]
	This implies
	\[ \lim_{t \to \infty} |{\dot \theta}_i(t)| = 0, \quad \forall~i \in {\mathbb N},  \]
	so that complete synchronization emerges:
	\[ \lim_{t \to \infty} |{\dot \theta}_i(t) - {\dot \theta}_j(t)| = 0, \quad  i, j \in {\mathbb N}.  \]
	
	\vspace{0.2cm}
	
	\noindent (ii)~From \eqref{E-7}, we have
	\[ r(t) \geq r(0) > 0 \quad \mbox{and} \quad  \lim_{t \to \infty} \sin \left(\theta_{k}(t) -\phi(t) \right) = \sin \Big(\lim_{t \to \infty} (\theta_{k}(t) -\phi(t) ) \Big)=  0, \quad \forall~k \in {\mathbb N}.  \]
	Hence, we have
	\[ \lim_{t \to \infty}  \left({\theta_{i}(t) -\theta_j(t)} \right) = n_{ij} \pi, \quad \mbox{for some $n_{ij} \in {\mathbb Z}$}. \]
\end{proof}
\subsubsection{Constant of motion} In this part, we provide two time-invariants for the dynamical system \eqref{E-1-1} that allow us to identify synchronized states.   \newline

\noindent $\bullet$~(Constant of motion I):~Let $\Theta$ be a phase configuration whose dynamics is governed by \eqref{E-1-1}. Then, the weighted sum ${\mathcal S}(\Theta, A)$ is given as follows:
\begin{equation*} \label{E-9}
	{\mathcal S}(\Theta, A) := \sum_{k\in\mathbb{N}} \kappa_k \theta_k, \quad {\boldsymbol{\kappa}}  = (\kappa_k).
\end{equation*}
Then, it is easy to see that ${\mathcal S}(\Theta, A)$ is time-invariant:
\begin{equation} \label{E-9-1}
	\frac{d}{dt} {\mathcal S}(\Theta, A) = \sum_{k\in\mathbb{N}} \kappa_k {\dot \theta}_k =    \sum_{j, k=1}^{\infty} \kappa_j \kappa_k \sin\left(\theta_j-\theta_{k}\right) = 0.
\end{equation}
In the following proposition, we are ready to verify the convergence of $\theta_i$ for each $i\in\mathbb{N}$. First we present the collision avoidance between oscillators.

\begin{lemma}\label{L5.2}
	Let $\Theta$ be a solution to \eqref{E-1-1}. Then for each $i,j\in\mathbb{N}$,
	\[
	\theta_j^{\text{in}}<\theta_{i}^{\text{in}} \quad \Longrightarrow \quad \theta_j(t)\le\theta_{i}(t)\le\theta_j(t)+2\pi\text{ for }~~t > 0.
	\]	
\end{lemma}
\begin{proof}
	Suppose that there exists a first collision time $t_0 >0$ between $\theta_i$ and $\theta_j$, i.e.,
	\begin{equation} \label{E-10}
		{ \theta_j(t) <  \theta_i(t)}, \quad t < t_0, \quad  \theta_j(t_0) = \theta_i(t_0). 
	\end{equation}
	Then, it follows from \eqref{E-1-1} that
	\begin{equation}\label{E-11}
		\dot{\theta}_{i}(t_0)-\dot{\theta}_{j}(t_0) =r \Big(\sin\left(\phi(t_0)-\theta_{i}(t_0)\right)-\sin\left(\phi(t_0)-\theta_j(t_0)\right) \Big) = 0. 
	\end{equation}
	Inductively, one can see that
	\[ 
	\frac{d^n \theta_i}{dt^n} \Big|_{t = t_0} = \frac{d^n \theta_j}{dt^n} \Big|_{t = t_0}, \quad n \geq 2. 
	\]	
	Since $\theta_i - \theta_j$ is analytic at $t = t_0$ by Proposition \ref{P2.2}, there exists $\delta > 0$ such that 
	\[ \theta_i(t) = \theta_j(t), \quad t \in (t_0 - \delta, t_0 + \delta), \]
	which is contradictory to $\eqref{E-10}_1$. 
\end{proof}
The result of Lemma \ref{L5.2} implies that if oscillators' phases are different initially, they can not cross each other in any finite time. On the other hand, Theorem \ref{T5.1} does not imply the convergence of each phase itself. By combining the conservation of the weighted sum, one can show that each oscillator is converging.

\begin{proposition}\label{P5.2}
	Let $\Theta$ be a solution to \eqref{E-1-1} with initial data satisfying the following conditions:
	\[  0<\theta_{i}^{\text{in}}<2\pi, \quad i \in {\mathbb N}. \]
	Then, there exists a constant state $\Theta^{\infty} = \{ \theta_i^{\infty} \}$ such that 
	\[ \lim_{t \to \infty} \theta_i(t)  = \theta_i^{\infty}, \quad  i \in {\mathbb N}. \]
\end{proposition}
\begin{proof}
	By Theorem \ref{T5.1} and Lemma \ref{L5.2},  one has 
	\begin{align*}
		\begin{aligned}
			& \left|\theta_{i}(t)-\theta_j(t)\right|\le2\pi, \quad \forall~i,j\in\mathbb{N} \quad \mbox{and} \\
			& \exists~\theta_{ij}^{\infty} \in (-2\pi, 2\pi)\quad \mbox{such that} \quad \lim_{t \to \infty} (\theta_i(t) - \theta_j(t)) = \theta_{ij}^{\infty}. 
		\end{aligned}
	\end{align*}
	On the other hand, note that 
	\begin{equation}\label{E-12}
		{
			\left(\sum_{i\in\mathbb{N}} \kappa_{i}\theta_{i}^{\text{in}}\right)-\theta_j\left(t\right)=
			\left(\sum_{i\in\mathbb{N}} \kappa_{i}\theta_{i}\left(t\right)\right)-\theta_j\left(t\right)=\sum_{i\in\mathbb{N}} \kappa_{i}\left(\theta_{i}\left(t\right)-\theta_j\left(t\right)\right).}
	\end{equation}
	Next, we show that the R.H.S. of \eqref{E-12} converges as $t \to \infty$. More precisely, we claim
	\begin{equation} \label{E-13}
		\lim_{t\to\infty}\sum_{i\in\mathbb{N}} \kappa_{i}\left(\theta_{i}\left(t\right)-\theta_j\left(t\right)\right)=\sum_{i\in\mathbb{N}} \kappa_{i}\theta_{ij}^{\infty}.
	\end{equation}
	{\it Proof of \eqref{E-13}}:~Since { $\displaystyle \sum_{i=1}^{\infty} \kappa_i = 1$}, for any $\varepsilon > 0$,  there exists a $n_{\varepsilon}\in\mathbb{N}$
	such that
	\begin{equation} \label{E-14}
		\sum_{i\ge n_{\varepsilon}} \kappa_{i}< \frac{\varepsilon}{4\pi}. 
	\end{equation} 
	For $i<n_{\varepsilon}$, we can choose $t_{\varepsilon}$ such that 
	\begin{equation} \label{E-15}
		t > t_\varepsilon \quad \Longrightarrow \quad \left|\theta_{i}\left(t\right)-\theta_j\left(t\right)-\theta_{ij}^{\infty}\right|< \frac{\varepsilon}{2}.
	\end{equation}
	Now, we use \eqref{E-14} and \eqref{E-15} to obtain
	\begin{align*}
		& \left|\sum_{i\in\mathbb{N}} \kappa_{i}\left(\theta_{i}\left(t\right)-\theta_j\left(t\right)\right)-{\sum_{i\in\mathbb{N}} \kappa_{i}\theta_{ij}^{\infty}}\right|\\
		& \hspace{1cm} \le\sum_{i<n_{\varepsilon}}\left|\kappa_{i}\left(\theta_{i}\left(t\right)-\theta_j\left(t\right)\right)- \kappa_{i}\theta_{ij}^{\infty}\right|+\sum_{i\ge n_{\varepsilon}}\left|\kappa_{i}\left(\theta_{i}\left(t\right)-\theta_j\left(t\right)\right)- \kappa_{i}\theta_{ij}^{\infty}\right|\\
		&\hspace{1cm} \le\sum_{i<n_{\varepsilon}} \kappa_{i} \frac{\varepsilon}{2} +\sum_{i\ge n_{\varepsilon}} \kappa_{i}\cdot 2\pi  \le\frac{\varepsilon}{2}+\frac{\varepsilon}{2}=\varepsilon
	\end{align*}
	for $t>t_{\varepsilon}$. Hence we verified \eqref{E-13}. Finally, it follows from \eqref{E-12} and \eqref{E-13} that 
	\[ \lim_{t \to \infty}  \theta_j(t) = \sum_{i\in\mathbb{N}} \kappa_{i}\theta_{i}^{\text{in}} -\sum_{i\in\mathbb{N}} \kappa_{i}\theta_{ij}^{\infty} =: \theta_j^{\infty}, \quad j \in {\mathbb N}.  \]
\end{proof}

\vspace{0.2cm}

\noindent $\bullet$~(Constant of motion II): As a second choice for the constant of motion, we consider a cross-ratio-like quantity for four distinct points on the unit circle. Before we discuss the second constant of motion, we recall the complex lifting of the Kuramoto model in \eqref{E-1-2}.  For this, we set a point on the unit circle associated with the phase $\theta_i$:
\[
z_i = e^{{\mathrm i} \theta_i}, \quad i \in {\mathbb N}. 
\]
Then, it is easy to check that the Kuramoto model $\eqref{E-1}_1$ can cast as follows.
\begin{equation} \label{E-1-2}
	\dot{z}_{i}=i\nu_{i}z_{i}+{\frac{1}{2}\sum_{j\in\mathbb{N}} \kappa_j\left(z_{j}-\left\langle z_{j},z_{i}\right\rangle z_{i}\right),   \quad \mbox{where $\left\langle z_{j},z_{i}\right\rangle =\bar{z}_{j}z_{i}$. }}
\end{equation}
We set 
\[  \omega(t) = \sum_{n=1}^{\infty} \kappa_n z_n(t). \]
\begin{lemma}	 \label{L5.3}
	Let $\{z_i \}$ be a solution to \eqref{E-1-2} such that 
	\[ z_i \neq z_j, \quad \forall~i \neq j, \quad \nu_i \equiv 0, \quad i \in {\mathbb N}. \]
	Then, we have the following relations: for any $i \neq j\in\mathbb{N}$,
	\begin{equation*} \label{E-1-3}
		\frac{d}{dt}\left(z_{i}-z_{j}\right)  =-\frac{1}{2}\overline{\omega} \left(z_{i}^{2}-z_{j}^{2}\right), \quad  
		\frac{d}{dt}\left(\frac{1}{z_{i}-z_{j}}\right) = \frac{{\bar \omega}}{2}\frac{\left(z_{i}^{2}-z_{j}^{2}\right)}{\left(z_{i}-z_{j}\right)^{2}}. 
	\end{equation*}	
\end{lemma}
\begin{proof} Note that \eqref{E-1-2} can be rewritten as 
	\[ {\dot z}_i = \frac{1}{2} \Big (\omega - {\bar \omega} z_i^2 \Big ).  \]
	This yields the desired estimates:
	\[ {\dot z}_i - {\dot z}_j = -\frac{{\bar \omega}}{2} (z_i^2 - z_j^2), \quad \frac{d}{dt}\left(\frac{1}{z_{i}-z_{j}}\right) = -\frac{1}{(z_i - z_j)^2} \frac{d}{dt} ({\dot z}_i - {\dot z}_j) = \frac{{\bar \omega}}{2} 
	\frac{z_i^2 - z_j^2}{(z_i - z_j)^2}.   \]
\end{proof}	
For $\{z_i :=  e^{\rm{i}\theta_i}\}_{i\in\mathbb{N}}$, we define a cross ratio-like functional ${\mathcal C}_{ijkl}$ as 
\begin{equation*} \label{E-16}
	{\mathcal C}_{ijkl}:=\frac{z_{i}-z_{k}}{z_{i}-z_{j}}\cdot\frac{z_{j}-z_{l}}{z_{k}-z_{l}}.
\end{equation*}
\begin{proposition} \label{P5.3}
	Let $\Theta$ be a solution to \eqref{E-1-1} with non-collisional initial data:
	\[ \theta_i^{\text{in}} \neq \theta_j^{\text{in}} \quad \mbox{for $i \neq j$}. \]
	Then, for any four-tuple $(i,j, k, l) \in {\mathbb N}^4$, ${\mathcal C}_{ijkl}$ is well-defined for all $t> 0$ and constant:
	\[ {\mathcal C}_{ijkl}(t) = {\mathcal C}_{ijkl}(0), \quad t > 0. \]
\end{proposition}
\begin{proof}
	Since all points $\{ e^{{\mathrm i} \theta_i} \}$ are distinct, ${\mathcal C}_{ijkl}$ is well-defined at $t = 0$. Moreover, by the continuity of solution, there exists $\eta > 0$ such that  for $i \neq j$, 
	\[ \theta_i(t) \neq \theta_j(t) \quad t \in (0, \eta). \]
	Thus, the cross-ratio like functional ${\mathcal C}_{ijkl}$ is well-defined in the time interval $(0, \eta)$. Now we introduce a temporal set ${\mathcal T}$ and its supremum $\tau^*$ :
	\[ {\mathcal T}:= \{ \tau \in (0, \infty)~:~{\mathcal C}_{ijkl}~\mbox{is well-defined in the time interval $(0, \tau)$} \}, \quad \tau^* := \mbox{sup} {\mathcal T}.  \]
	Then, the set ${\mathcal T}$ is not empty and $\tau^{*} \in (0, \infty]$.  In what follows, we show that 
	\begin{equation*} \label{E-17}
		\tau^* = \infty \quad \mbox{and} \quad  {\mathcal C}_{ijkl}(t) = {\mathcal C}_{ijkl}(0), \quad t > 0. 
	\end{equation*}
	Suppose that the contrary holds, not, i.e., 
	\[ \tau^* < \infty. \]
	First, we show that the functional ${\mathcal C}_{ijkl}$ is constant in the interval $(0, \tau^*)$. For this,  we use Lemma \ref{L5.3} to get
	\begin{align*}
		\begin{aligned}
			\frac{d}{dt} {\mathcal C}_{ijkl}(t) &=-\left(\frac{1}{2}\overline{\omega}(t)\left(z_{i}(t)+z_{k}(t)\right)+\frac{1}{2}\overline{\omega}(t)\left(z_{j}(t)+z_{l}(t)\right)\right)C_{ijkl}(t)\\
			&\phantom{=}+\left(\frac{1}{2}\overline{\omega}(t)\left(z_{i}(t)+z_{j}(t)\right)+\frac{1}{2}\overline{\omega}(t)\left(z_{k}(t)+z_{l}(t)\right)\right)C_{ijkl}(t)\\
			& =0, \qquad t \in (0, \tau^*).
		\end{aligned}
	\end{align*}
	Thus, as long as  ${\mathcal C}_{ijkl}$ is well-defined, it is constant.  Certainly, it is continuous with respect its arguments. Therefore, 
	\[ \exists~{\mathcal C}_{ijkl}(\tau^*) = \lim_{t \to \tau^*-} {\mathcal C}_{ijkl}(t). \]
	By continuity,  there exists a $\delta > 0$ such that 
	\[ {\mathcal C}_{ijkl}(\cdot)~\mbox{is well-defined in the time-interval $[0, \tau^* + \delta)$} \]
	which is contradictory to the choice of $\tau^*$. Therefore we have
	\[ \tau^* = \infty, \]
	i.e.,  ${\mathcal C}_{ijkl}(\cdot)$ is well-defined on the whole time interval $[0, \infty)$ and  
	\[ {\mathcal C}_{ijkl}(t) =  {\mathcal C}_{ijkl}(0).  \quad t \in (0, \infty). \]

\end{proof}
As a direct corollary of Proposition \ref{P5.3}, we have the following results on the asymptotic configurations of the set $\{ e^{{\mathrm i} \theta_i} \}$ and $\{ \theta_i \}$. First, we will see that asymptotic configuration of  $\{ e^{{\mathrm i} \theta_i} \}$ is either a singleton or bi-polar configuration. 
\begin{corollary} \label{C5.1}
	Let $\{z_i \}$ be a solution to \eqref{E-1-2} with asymptotic configuration $\{z_i^{\infty} \}$. Then, the following dichotomy holds.
	\begin{enumerate}
		\item There exists a $k\in\mathbb{N}$ such that $ z_{i}^{\infty}=-z_{k}^{\infty}$ for $i\in\mathbb{N}\setminus\{ k\}$.
		\vspace{0.2cm}
		\item $z_{i}^{\infty}=z_{j}^{\infty}$ for all $i, j \in {\mathbb N}$. 
	\end{enumerate}
\end{corollary}
\begin{proof}
	Suppose that there exists a $1 \neq k\in\mathbb{N}$ such that
	\[  z_1^\infty\neq z_k^\infty. \]
	By { Theorem \ref{T5.1}} and Proposition \ref{P5.2}, $\theta_i^\infty - \theta_k^\infty$ is an integer multiple of $\pi$, which implies $z_1^\infty =  -z_k^\infty$. Then we set a partition $I_{1}\cup I_{2}$ of $\mathbb{N}$ by
	\[
	I_{1}  :=\left\{ i \in {\mathbb N} \ |\ z_{i}\rightarrow z_{1}^{\infty}\text{ as }t\rightarrow\infty\right\}, \quad  I_{2} :=\left\{ i \in {\mathbb N}\ |\ z_{i}\rightarrow-z_{1}^{\infty}\text{ as }t\rightarrow\infty\right\}. \]
	Suppose that 
	\[ |I_1| \geq 2 \quad \mbox{and} \quad  |I_2| \geq 2. \]
	Then, we can choose 
	\[ i \neq j\in I_{1} \quad \mbox{and} \quad  k \neq l \in I_2. \]
	For such pairs $(i,j)$ and $(k,l)$, 
	\[ \lim_{t \to \infty} |{\mathcal C}_{ijkl}(t) |=\lim_{t \to \infty} \Big| \frac{z_{i}(t)-z_{k}(t)}{z_{i}(t)-z_{j}(t)}\cdot\frac{z_{j}(t)-z_{l}(t)}{z_{k}(t)-z_{l}(t)}  \Big| =  \infty, \]
	which is contradictory to the constancy of ${\mathcal C}_{ijkl}$:
	\[ {\mathcal C}_{ijkl}(t) = {\mathcal C}_{ijkl}(0), \quad t > 0. \]
	Therefore, we have
	\[ |I_1| \leq 1 \quad \mbox{and} \quad  |I_2|  \leq 1. \]
	Without loss of generality, we may assume $I_1\leq 1$. Then there are two cases: \newline
	
	\noindent If $|I_1| = 0$, then asymptotic state is in complete phase synchrony:
	\[ \lim_{t \to \infty}  z_i(t)  = z_1^{\infty}, \quad i \geq 2. \]
	If $|I_1| = 1$, then we have bi-polar asymptotic state:
	\[ \lim_{t \to \infty} z_i(t) = -z_1^{\infty}, \quad i \geq 2. \]
\end{proof}
Now we are ready to study the asymptotic configuration of \eqref{E-1-1} that emerges from the given initial configuration $\{ \theta_i^{\text{in}} \}$.  For a given initial configuration  $\{ \theta_i^{\text{in}} \}$, we set
\begin{equation} \label{E-18}
	\theta_0 :=  \sum_{i\in\mathbb{N}} \kappa_{i}\theta_{i}^{\text{in}}.
\end{equation}
Then, it follows from \eqref{E-9-1} that 
\begin{equation} \label{E-19}
	\theta_0 = \sum_{i\in\mathbb{N}} \kappa_{i}\theta_{i}(t), \quad t > 0.
\end{equation}
\begin{corollary}  \label{C5.2}
	Let $\Theta$ be a solution to \eqref{E-1-1} with asymptotic configuration $\{\theta_i^{\infty} \}$:
	\[ \lim_{t \to \infty} \theta_i(t) = \theta_i^{\infty}, \quad   i \in {\mathbb N}. \]
	Then, for each $i\in\mathbb{N}$, 
	\[
	\theta_{i}^{\infty}\in\left\{ \theta_{0}\right\} \cup\left\{ \theta_{0}\pm \kappa_{i}\pi\ |\ i\in\mathbb{N}\right\} \cup\left\{ \theta_{0}\pm\left(1-\kappa_{i}\right)\pi\ |\ i\in\mathbb{N}\right\} .
	\]
\end{corollary}
\begin{proof}  It follows from { Corollary \ref{C5.1}} that we have two possible asymptotic configurations:
	\begin{center}
		Complete phase synchrony and bi-polar configuration.
	\end{center}
	\noindent $\bullet$~Case A:~Suppose that 
	\[ \lim_{t \to \infty} \theta_i(t) = \theta_\infty, \quad  \forall~i \in {\mathbb N}. \]
	In this case, we use the above relation, \eqref{E-18}  and \eqref{E-19} to get
	\[
	\theta_{0}=\sum_{i\in\mathbb{N}} \kappa_{i}\theta_{i}(t)= \lim_{t \to \infty}  \sum_{i\in\mathbb{N}} \kappa_{i}\theta_{i}(t) = \sum_{i\in\mathbb{N}} \kappa_{i}\theta_{i}^{\infty}=\sum_{i\in\mathbb{N}} \kappa_{i}\theta_{\infty}=\theta_{\infty}.
	\]
	
	\vspace{0.2cm}
	
	\noindent $\bullet$~Case B:~Suppose that 
	\begin{equation} \label{E-20}
		\lim_{t \to \infty} \theta_j(t) = \theta_{\infty}\pm\pi \quad \mbox{for some $j$}  \quad \mbox{and} \quad  \lim_{t \to \infty} \theta_{i}(t) = \theta_{\infty} \quad \mbox{for all  $i\neq j$}. 
	\end{equation}
	Then, we use the above relations and the same idea as Case A to find
	\[
	\theta_{0}=\sum_{i\in\mathbb{N}} \kappa_{i}\theta_{i}(t)=\sum_{i\in\mathbb{N}} \kappa_{i}\theta_{i}^{\infty}=\sum_{k\in\mathbb{N}} \kappa_{k}\theta_{\infty}\pm \kappa_j\pi=\theta_{\infty}\pm \kappa_j\pi.
	\]
	This and \eqref{E-20} imply
	\[
	\theta_{i}\to\theta_{\infty}=\theta_{0}\mp \kappa_j\pi,\quad \theta_j\rightarrow\theta_{0}\pm\left(1-\kappa_j\right)\pi.
	\]
	Finally, we combine all the results in Case A and Case B to obtain the desired estimate. 
\end{proof}

\subsection{Heterogeneous ensemble} \label{sec:5.2}
In this subsection, we study the frequency synchronization of the heterogeneous ensemble for a restricted class of initial configurations confined in a half circle.  \newline

Note that the Cauchy problem \eqref{E-1} is equivalent to the following Cauchy problem:
\begin{equation}\label{E-21}
	\begin{cases}
		\displaystyle	\dot{\theta}_{i}=\omega_{i},\quad t>0,\quad\forall~ i\in\mathbb{N},\\
		\displaystyle	\dot{\omega}_{i}=\sum_{j\in\mathbb{N}} \kappa_j\cos\left(\theta_i-\theta_j\right)\left(\omega_j-\omega_{i}\right), \\
		\displaystyle	\theta_{i}(0)=\theta_{i}^{\text{in}}\in\mathbb{R},\quad\omega_{i}(0)=\nu_{i}+\sum_{j\in\mathbb{N}} \kappa_j\sin\left(\theta_j^{\text{in}}-\theta_{i}^{\text{in}}\right),
	\end{cases}
\end{equation}
where 
\[ \Theta^{\text{in}} = (\theta_1^{\text{in}}, \theta_2^{\text{in}}, \ldots ) \in \ell^{\infty}, \quad {\mathcal V} = (\nu_1, \nu_2, \ldots) \in \ell^{\infty}. \]
We set 
\[  {\mathcal W} := (\omega_1, \omega_2,  \ldots) \quad \mbox{and} \quad  {\mathcal D}({\mathcal W}) : = \sup_{m, n} |\omega_m - \omega_n|.  \]
Note that for our sender network, \[{\displaystyle \left\Vert K\right\Vert _{\infty,1}=\left\Vert K\right\Vert _{-\infty,1}=\sum_{j\in\mathbb{N}}\kappa_{j}}.\] Here the Theorem \ref{T4.1} can be applied to trap $\Theta(t)$ into a quarter arc.

\begin{proposition}\label{P5.4}
	Suppose that the initial condition $\mathcal{D}(\Theta^{\text{in}})$
	and network topology $\left(\kappa_{ij}\right)$ satisfy
	\[
	0<\mathcal{D}(\mathcal{V})<\|\boldsymbol{\tilde{\kappa}}_{\epsilon}\|_{1}\|K\|_{\infty,1},\quad\mathcal{D}(\Theta^{\text{in}})\in(\gamma,\pi-\gamma),\quad\gamma=\sin^{-1}\left(\frac{\mathcal{D}(\mathcal{V})}{\|\boldsymbol{\tilde{\kappa}}_{\epsilon}\|_{1}\|K\|_{\infty,1}}\right)<\frac{\pi}{2}
	\]
	for 
	\[ \varepsilon\ll 1,\quad
	\tilde{\kappa}_{\varepsilon,i}=\frac{\kappa_{i}}{\left\Vert K\right\Vert _{\infty,1}+\varepsilon},\quad \boldsymbol{\tilde{\kappa}}_{\varepsilon}=\left\{\tilde{\kappa}_{\varepsilon,i}\right\}_{i\in\mathbb{N}}.
	\]
	Then there exists a $t_{0}>0$ such that 
	\[
	\mathcal{D}\left(\Theta(t)\right)<\frac{\pi}{2}-\sin^{-1}\delta,\quad t\ge t_{0}.
	\]
\end{proposition}
\begin{proof}
	The conclusion is straightforward from Theorem \ref{T4.1}.
\end{proof}

Next, we state our last main results on the complete synchronization of \eqref{E-21}. 

\begin{theorem} \label{T5.2}
	Let $(\Theta, {\mathcal W})$ be a solution to \eqref{E-21} with conditions in Proposition \ref{P5.4}. Then, there exists $t_0 \geq 0$ such that 
	\[  {{\mathcal D}(\Theta(t))<\frac{\pi}{2}}, \quad  {\mathcal D}({\mathcal W}(t)) \le {\mathcal D}({\mathcal W}(t_{0})) \cdot\exp\left[ -\frac{3\| A \|_{\infty,1} \log2}{32}\left(t-t_{0}\right)+1\right ],\quad t\ge t_{0}. \]
\end{theorem}
\begin{proof} Since the proof is very lengthy, we leave its proof in Appendix \ref{App-C}.
\end{proof}

\section{Conclusion}\label{sec:6} 
\setcounter{equation}{0} 
In this paper, we have proposed a generalized synchronization model for the set of the infinite set of Kuramoto oscillators and studied its emergent asymptotic dynamics. The original Kuramoto model describes the synchronous dynamics of a finite set of Kuramoto oscillators, and it has been extensively studied in the last decade, whereas as far as the authors know, the dynamics of an infinite number of Kuramoto oscillators have not been addressed in literature as it is. In fact, for the dynamics of an infinite ensemble,  the Kuramoto-Sakaguchi equation which is corresponding to a mean-field approximation is often used to describe the temporal-phase space dynamics of a one-particle distribution function.  However, this is only an approximation for the dynamics of the infinite set of Kuramoto oscillators. To make sense of the infinite coupling term, we need to introduce a suitable coupling weight. These coupling weights can be realized as a  network topology which is represented by an infinite matrix whose row sums are finite uniformly.  For this set of an infinite number of oscillator equations, some finite-dimensional results and analytical tools can be used for our infinite setting, but there { exists} some fundamental differences.  For example,  the Dini derivative of phase diameter cannot be used, because we cannot estimate how many crossings occur at the end-point of phase space. Similarly, the gradient flow structure of the Kuramoto model cannot be applied even for a symmetric network. We show that there exists a network topology that leads to the constancy of phase diameter (see Corollary \ref{C3.1}).  

For a symmetric network topology and homogeneous ensemble with the same natural frequency, we show that complete synchronization occurs asymptotically (see Theorem \ref{T3.1}), whereas for a heterogeneous ensemble, we cannot show complete synchronization, but we instead obtain a practical synchronization, i.e., we can make the size of phase diameter as small as we want by increasing the size of coupling strength (see Theorem \ref{T4.1}). On the other hand, for a sender network topology in which coupling strength depends on neighboring oscillators, a homogeneous ensemble either evolves toward complete phase synchrony or a special type of bi-cluster configuration (see Theorem \ref{T5.1}). In contrast, for a heterogeneous ensemble, we show that complete synchronization emerges asymptotically (see Theorem \ref{T5.1}). There are several interesting remaining questions. For example, the relation between finite collisions and phase-locking is not clear at all. For the Kuramoto model for a finite ensemble, the aforementioned relations are equivalent. Moreover, we did not show the emergence of complete synchronization for a heterogeneous ensemble in a large coupling regime. We leave these interesting questions as future work. 

\backmatter

\bmhead{Acknowledgments}

The work of S.-Y. Ha was supported by National Research Foundation of Korea (NRF-2020R1A2C3A01003881)

\newpage

\begin{appendices}

\section{Some useful lemmas} \label{App-A}
\setcounter{equation}{0}
In this appendix, we collect some useful results which are used explicitly and implicitly in the main body of the paper without detailed explanation and proofs. Detailed proofs can be found in quoted references and any reasonable book on mathematical analysis, e.g. \cite{Ru}. \newline

First, we begin with the abstract Cauchy problem on a Banach space $(E, \| \cdot \|)$:
\begin{equation} \label{Ap-1}
	\begin{cases}
		\frac{du}{dt}=F\left(u\left(t\right)\right), \quad t > 0, \\
		u(0) =u^{\text{in}}.
	\end{cases}
\end{equation}
\begin{lemma}  \label{LA-1}
	\emph{(Global well-posedness \cite{Bre, C})}
	The following assertions hold.
	\begin{enumerate}
		\item
		\emph{(Existence)}: Let $F:E\rightarrow E$ be a Lipschitz
		map, i.e. there is a nonnegative constant $L$ such that 
		\[
		\left|\left|Fu-Fv\right|\right|\le L\left|\left|u-v\right|\right|\quad \forall~ u,v\in E.
		\]
		Then, for any  given $u^{\text{in}} \in E$, there exists a global solution $u \in {\mathcal C}^{1} ([[0, \infty);E)$ to \eqref{Ap-1}. 
		\vspace{0.2cm}
		\item
		\emph{(Uniqueness)}: For $U\subset E$, let $F:~U\to E$ be a locally Lipschitz map; let $I$ be an interval contained in $\mathbb{R}$
		not necessarily compact. If there are two exact local solutions $\varphi_{1}$
		and $\varphi_{2}$ $:I\to E$ to \eqref{Ap-1}.  Then, they are identical in the entire interval $I$.
	\end{enumerate}	
\end{lemma}
\begin{remark}
	This lemma has been used to guarantee the global well-posedness of the infinite Kuramoto model on the Banach space $(\ell^\infty,~\| \cdot \|_{\infty})$ in Proposition \ref{P2.2}. 
\end{remark}
Next, we present a differential version of Barbalat's lemma which has been used in the proof of Proposition \ref{P5.1}. 
\begin{lemma} \label{LA-2}
	\emph{(Barbalat \cite{B})} Let  $F: [0, \infty) \rightarrow  \mathbb{R}$ be a continuously differentiable function satisfying the following two properties:
	\[ \exists~\lim_{t \to \infty} F(t) \quad \mbox{and} \quad F^{\prime}~\mbox{is uniformly continuous}. \]
	Then, $F^{\prime}$ tends to zero, as $t \to \infty$. 
\end{lemma}
\begin{lemma} \label{LA-3}
	\emph{\cite{Ru} } 
	Let $F_{n}$ be a sequence of real-valued differentiable functions on $\left[a,b\right]$ with the following two properties:
	\begin{enumerate}
		\item
		(Pointwise convergence at one-point):  For some $x_0 \in [a, b]$, 
		\[  \exists~ \lim_{n \to \infty} F_n(x_0). \]
		\item
		(Uniform convergence of derivatives):  the sequence $\left\{ F_{n}^{\prime} \right\} $ converges uniformly on $\left[a,b\right]$.
	\end{enumerate}
	Then, $\left\{ F_{n}\right\} $converges uniformly on $\left[a,b\right]$ to a function $F$ and
	\[
	F^{\prime}(x) =\lim_{n\rightarrow\infty}F^{\prime}_{n}(x)\quad\left(a\le x\le b\right).
	\]
\end{lemma}
\begin{remark}
	The proof can be found in Theorem 7.17 of \cite{Ru}. 
\end{remark}
Finally, we state the Lojasiewicz gradient inequality  on a Hilbert space  $(H, \langle \cdot, \cdot \rangle)$: We set 
\[ \| u \| := \sqrt{\langle u, u \rangle}, \quad u \in H. \]
\begin{lemma}  \label{LA-4}
	\emph{(Lojasiewicz gradient inequality  \cite{H-J})} 
	For an open neighborhood $U \subset H$ of $0 \in H$, let $F:U \to {\mathbb R}$ 
	be an analytic function such that 
	\[  F(0)=0, \quad  DF(0)=0. \]
	Suppose $F$ satisfies the following two conditions:
	\begin{enumerate}
		\item
		${\mathcal N}:= {\ker}(D^{2}F(0))$ is finite-dimensional.
		\item
		There is $\rho>0$ such that 
		\begin{equation} \label{New-A}
			\| D^{2}F(0)u \| \ge \rho \| u \|, \quad \forall~u\in V\cap N^{\perp},
		\end{equation}
		where ${\mathcal N}^{\perp}$stands for the orthogonal complement of ${\mathcal N}$.
	\end{enumerate}	
	Then there exist $\theta\in\left(0,1/2\right)$, a neighborhood $W$
	of $0$ and $c>0$ such that
	\[
	\| DF(u) \| \ge c \Big| F(u) \Big|^{1-\theta}, \quad \forall~u \in W.
	\]
\end{lemma}
\begin{remark}
	See the discussions right after the proof of Theorem \ref{T3.1}. 
\end{remark}

\vspace{0.5cm}
%%%%%%%%%%%%%%%%%%%%
%
%
%
%%%%%%%%%%%%%%%%%%%%%%%%%%

\section{Proof of Lemma \ref{L3.1}} \label{App-B}
\setcounter{equation}{0}
In this appendix, we provide a lengthy proof of Lemma \ref{L3.1}. Since $\displaystyle\overline{\theta}(0) = \sup_{i \in {\mathbb N}} \theta_i^{\text {in}}$, the following dichotomy holds: 
\begin{eqnarray*}
	&& (1)~\mbox{$\overline{\theta}(0)$ is an isolated point of the set $\{ \theta_i^{\text {in}}\}_{i \in {\mathbb N}}$,} \\
	&&(2) ~\mbox{$\overline{\theta}(0)$ is a limit point of the set $\{ \theta_i^{\text {in}} \}_{i \in {\mathbb N}}$.}
\end{eqnarray*}

\vspace{0.2cm}

\noindent In the sequel, we show that the desired assertions hold for each case. \newline

\noindent (1)~First of all, suppose  that 
\begin{center}
	$\overline{\theta}(0)$ is not a limit point of the set  $\{ \theta_i^{\text {in}}\}_{i \in {\mathbb N}}$.
\end{center}
In this case,  the index set 
\[ {\mathcal I}_{\overline{\theta}(0)}:=\left\{i\in \mathbb{N}:\theta_i^{\text {in}}= \overline{\theta}(0) \right\}\]
is nonempty (possibly infinite) subset of $\mathbb{N}$, and there exists $\varepsilon>0$ such that 
\[  {\mathcal D}(\Theta^{\text {in}}) + \varepsilon < \pi \quad \mbox{and} \quad  {\mathcal I}_{( \overline{\theta}(0)-\varepsilon,  \overline{\theta}(0))} :=\left\{i \in {\mathbb N}:~\theta_i^{\text {in}}\in ( \overline{\theta}(0)-\varepsilon,  \overline{\theta}(0)) \right\}=\emptyset. \]
\ \\
\noindent $\diamond$~Case A.1:  Let $i \in {\mathbb N}$ be an index such that  
\[  \theta_i^{\text {in}}\leq  {\overline \theta}(0) - \varepsilon. \]
We set 
\[
\delta_1 :=  \frac{\varepsilon}{\|K \|_{\infty,1}} > 0.
\]
For $t \in (0, \delta_1)$, we use the above relation and Lemma \ref{L2.2} to obtain 
\begin{equation} \label{C-1-1-0}
	\theta_i(t)\leq \theta_i^{\text {in}}+ \| K \|_{\infty, 1} t <   {\overline \theta}(0) -\varepsilon+ \| K \|_{\infty,1} \cdot \frac{\varepsilon}{\|K \|_{\infty,1}}  =  {\overline \theta}(0).
\end{equation}

\noindent $\diamond$~Case A.2:  Let $i \in {\mathbb N}$ be an index such that 
\[ \theta_i^{\text {in}}>  \overline{\theta}(0) -\varepsilon. \]
In this case, it is easy to see that $\theta_i^{\text {in}}= \overline{\theta}(0)$ and
\begin{align}
	\begin{aligned} \label{C-1-1-1}
		\frac{d^+}{dt} \Big|_{t = 0+} \theta_i  &= \sum_{j\in\mathbb{N}} \kappa_{ij}\sin(\theta_j^{\text {in}}-\theta_i^{\text {in}}) =\sum_{j\in\mathbb{N}} \kappa_{ij}\sin(\theta_j^{\text {in}}-  \overline{\theta}(0)) \\
		&=  \sum_{j\in {\mathcal I}_{\overline{\theta}(0)}} \kappa_{ij} \sin(\underbrace{\theta_j^{\text {in}}-  \overline{\theta}(0)}_{=0}) +  \sum_{j \notin {\mathcal I}_{\overline{\theta}(0)}}\kappa_{ij}\sin(\theta_j^{\text {in}}-  \overline{\theta}(0)) \\
		&= \sum_{j\notin {\mathcal I}_{\overline{\theta}(0)}}\kappa_{ij}\sin(\theta_j^{\text {in}}-  \overline{\theta}(0) ) \leq - \left(\sum_{j\notin {\mathcal I}_{\overline{\theta}(0)}}\kappa_{ij}\right)\sin\varepsilon ,
	\end{aligned}
\end{align}
where $\frac{d^+}{dt}$ is the Dini derivative, and we used  \eqref{C-1}, \eqref{C-1-1-1} and the relation:
\[  j \notin {\mathcal I}_{\overline{\theta}(0)} \quad \Longrightarrow \quad -\pi  + \varepsilon <  \theta_j^{\text {in}} -  \overline{\theta}(0) \leq -\varepsilon.   \]
%there is an infinite sequence $\{i_k\}_{k\in\mathbb{N}}$ such that
%\[i_1<i_2<\cdots,\quad \lim_{k\to\infty}\theta_{i_k}(t_0)=M(t_0). \]
%$\bullet$ Case 1 (there is no such $\{i_k\}_{k\in\mathbb{N}}$): In this case, there exists $\varepsilon>0$ such that
%\[0<\left|\left\{i:\theta_i(t_0)=M(t_0) \right\} \right|<\infty,\quad \left|\left\{i:\theta_i(t_0)\in (M(t_0)-\varepsilon,M(t_0)) \right\} \right|=0. \]
On the other hand,  for $i \in {\mathcal I}_{\overline{\theta}(0)}$ and sufficiently small $t$ satisfying 
\begin{equation} \label{C-1-1-2}
	t < \delta_2:= \left(\sum_{j\notin {\mathcal I}_{\overline{\theta}(0)}}\tilde{\kappa}_j\right)\cdot \frac{\sin\varepsilon}{2 \| K \|_{\infty,1}}, 
\end{equation}
we apply \eqref{B-12} and \eqref{C-1-1-1} to obtain
\begin{align}
	\begin{aligned} \label{C-1-1-3}
		\dot{\theta}_i(t) &\leq \dot{\theta}_i(0)+  \Big( 2\|K \|_{\infty, 1}   \sum_{j\in\mathbb{N}} \kappa_{ij}  \Big) t  \\
		&<  \underbrace{\Bigg [  -\sum_{j\notin {\mathcal I}_{\overline{\theta}(0)}}\kappa_{ij} +   \Big( \sum_{j\in\mathbb{N}} \kappa_{ij} \Big) \Big(\sum_{j\notin {\mathcal I}_{\overline{\theta}(0)}}\tilde{\kappa}_j \Big) \Bigg]}_{<~ 0 \quad \mbox{from}~(\mathcal{F}2)}  \sin\varepsilon <0.
	\end{aligned}
\end{align}
Note that $\displaystyle\sum_{j\notin {\mathcal I}_{\overline{\theta}(0)}}\tilde{\kappa}_j$ is always strictly positive except the trivial case ${\mathcal I}_{\overline{\theta}(0)}=\mathbb{N}$, which we have already excluded by using the condition $\mathcal{D}(\Theta^{\text {in}})>0$. This guarantees the positivity of the constant $\delta_2$ in \eqref{C-1-1-2}. Finally, we set 
\begin{equation*}
	\delta := \min \{ \delta_1, \delta_2 \}  > 0,
\end{equation*}
and combine \eqref{C-1-1-0} and \eqref{C-1-1-3} to conclude the desired result. 

\vspace{0.5cm}

\noindent (2)~Suppose  that 
\begin{center}
	$\overline{\theta}(0)$ is a limit point of the set  $\{ \theta_i^{\text {in}}\}_{i \in {\mathbb N}}$.
\end{center}
In this case, we can exclude the singleton case with $\{ \theta_i^{\text {in}}\}_{i \in {\mathbb N}} =\left\{\overline{\theta}(0)\right\}$,  since every neighborhood of the limit point $\overline{\theta}(0)$ must contain a point of $\{ \theta_i^{\text {in}}\}_{i \in {\mathbb N}}$ other than $\overline{\theta}(0)$ itself. Therefore, there exists a natural number $i_0$ such that 
\[ \overline{\theta}(0)-\theta_{i_0}^{\text {in}}=:\varepsilon_0\in (0,\pi), \] so that the index set 
\[ {\mathcal I}_{[\underline{\theta}(0), \overline{\theta}(0)-\varepsilon_0]} :=\left\{i:\underline{\theta}(0) \leq \theta_i^{\text {in}}\leq \overline{\theta}(0)-\varepsilon_0  \right\} \]
is nonempty. Now, we define an auxiliary function $f:(0,\varepsilon_0)\to \mathbb{R}$ which will appear in \eqref{C-1-1-6}:
\[f(x) \equiv \left(\sum_{ j \in  {\mathcal I}_{[\underline{\theta}(0), \overline{\theta}(0)-\varepsilon_0]} }\tilde{\kappa}_j\right)\frac{\sin(\varepsilon_0-x)}{\sin x},\quad \forall~x\in (0,\varepsilon_0). \]
Then, it is easy to see that $f$ is a monotone decreasing continuous function such that 
\[  f^{\prime}(x)  < 0, \quad x \in (0, \varepsilon_0), \quad  \lim_{x\to 0+}f(x)=+\infty,\quad \lim_{x\to \varepsilon_0-}f(x)=0. \]
Therefore, there exists a unique $\varepsilon\in (0,\varepsilon_0)$ such that 
\begin{equation} \label{C-1-1-5}
	f(\varepsilon)=2, \quad \mbox{i.e.,} \quad \left(\sum_{ j \in  {\mathcal I}_{[\underline{\theta}(0), \overline{\theta}(0)-\varepsilon_0]} }\tilde{\kappa}_j\right) \sin(\varepsilon_0-\varepsilon) = 2 \sin \varepsilon.   
\end{equation}
\noindent $\diamond$~Case B.1: Let $i\in\mathbb{N}$ be an index such that 
\[  \theta_i^{\text {in}}\leq {\overline \theta}(0) -\varepsilon. \]
Then, for every positive $ t<\delta_1 = \frac{\varepsilon}{\|K \|_{\infty, 1}}  $, we use Lemma \ref{L2.2} to see
\[\theta_i(t)\leq \theta_i^{\text {in}}+ \| K \|_{\infty, 1} t < \overline{\theta}(0) -\varepsilon+ \| K \|_{\infty, 1}  \delta_1 = \overline{\theta}(0). \]

\noindent $\diamond$~Case B.2:~ Let $i \in {\mathbb N}$ be an index such that 
\[ \theta_i^{\text {in}}>  \overline{\theta}(0) -\varepsilon. \]
Note that the whole index set ${\mathbb N}$ can be rewritten as 
\begin{align*}
	\begin{aligned}
		{\mathbb N} &= \{i~:~\underline{\theta}(0) \leq \theta_i^{\text {in}}  \leq \overline{\theta}(0) - \varepsilon_0  \} \bigcup \{i:~ \overline{\theta}(0) - \varepsilon_0 < \theta_i^{\text {in}}  \leq \overline{\theta}(0) - \varepsilon \} \\
		&\hspace{0.5cm} \bigcup  \{ i:~\overline{\theta}(0) - \varepsilon < \theta_i^{\text {in}} \leq \overline{\theta}(0) \} \\
		&=: {{\mathcal I}_{[\underline{\theta}(0),  \overline{\theta}(0) - \varepsilon_0]}   \bigcup {\mathcal I}_{(\overline{\theta}(0) -\varepsilon_0,  \overline{\theta}(0) - \varepsilon]}  \bigcup  {\mathcal I}_{(\overline{\theta}(0) - \varepsilon,  \overline{\theta}(0)]}}.
	\end{aligned}
\end{align*}
Then, we have 
\begin{align}
	\begin{aligned} \label{C-1-1-6}
		&\frac{d^+\theta_i}{dt} \Big|_{t = 0+} \\
		&\hspace{0.5cm} { = \sum_{j \in  {\mathcal I}_{[\underline{\theta}(0),  \overline{\theta}(0) - \varepsilon_0]} }  \kappa_{ij} \sin(\theta_j^{\text {in}} - \theta_i^{\text {in}})  +   \sum_{j \in  {\mathcal I}_{(\overline{\theta}(0) -\varepsilon_0,  \overline{\theta}(0) - \varepsilon]} }  \kappa_{ij} \underbrace{\sin(\theta_j^{\text {in}} - \theta_i^{\text {in}})}_{\leq 0}} \\
		& \hspace{0.5cm} +  \sum_{j \in  {\mathcal I}_{(\overline{\theta}(0) - \varepsilon,  \overline{\theta}(0)]}}  \kappa_{ij} \sin(\theta_j^{\text {in}} - \theta_i^{\text {in}})         \\
		&\hspace{0.5cm}  \leq  \sum_{j \in  {\mathcal I}_{[\underline{\theta}(0),  \overline{\theta}(0) - \varepsilon_0]} } \kappa_{ij}  \sin(\theta_j^{\text {in}} - \theta_i^{\text {in}})  +   \sum_{j \in  {\mathcal I}_{(\overline{\theta}(0) - \varepsilon,  \overline{\theta}(0)]}}  \kappa_{ij}\sin(\theta_j^{\text {in}} - \theta_i^{\text {in}})  \\
		&\hspace{0.5cm}  \leq -\sum_{j \in  {\mathcal I}_{[\underline{\theta}(0),  \overline{\theta}(0) - \varepsilon_0]} } \kappa_{ij}  \sin(\varepsilon_0-\varepsilon)  +   \sum_{j \in  {\mathcal I}_{(\overline{\theta}(0) - \varepsilon,  \overline{\theta}(0)]}}  \kappa_{ij}\sin \varepsilon  \\
		&\hspace{0.5cm}  \leq -\sum_{j \in  {\mathcal I}_{[\underline{\theta}(0),  \overline{\theta}(0) - \varepsilon_0]} } \kappa_{ij}  \sin(\varepsilon_0-\varepsilon)  +   \sum_{j \in {\mathbb N}}  \kappa_{ij}\sin \varepsilon  \\
		&\hspace{0.5cm}  \leq \sum_{j\in \mathbb{N}} \kappa_{ij}\left(\sin\varepsilon- \sum_{j \in  {\mathcal I}_{[\underline{\theta}(0),  \overline{\theta}(0) - \varepsilon_0]} } \tilde{\kappa}_j\sin(\varepsilon_0-\varepsilon) \right)  \quad \mbox{from}~ (\mathcal{F}1)~\mbox{and}~\eqref{C-1-1-5} \\
		&\hspace{0.5cm}  =-\left(\sum_{j\in \mathbb{N}} \kappa_{ij}\right)\sin\varepsilon,
	\end{aligned} 
\end{align}
where we used \eqref{C-1-1-5} in the last equality. Now, we set 
\[ {\tilde \delta}_2 :=  \frac{\sin \varepsilon}{2 \| K \|_{\infty, 1}}. \]
Then,  for $t <{\tilde \delta}_2$,  we apply Lemma \ref{L2.2} as in Case A.2 to \eqref{C-1-1-6} to obtain
\begin{align*}
	\begin{aligned}
		\dot{\theta}_i(t) & \leq  \dot{\theta}_i(0)+ \Big( 2  \|K \|_{\infty, 1}  \sum_{j\in \mathbb{N}}\kappa_{ij} \Big) t < \dot{\theta}_i(0)+\left(\sum_{j\in \mathbb{N}} \kappa_{ij}\right)\sin\varepsilon\leq 0.
	\end{aligned}
\end{align*}
Finally, we define $\delta = \min\{ \delta_1, {\tilde \delta}_2 \}$ to get the desired result when $\overline{\theta}(0)$ is a limit point. \hfill$\qed$

\vspace{1cm}

\section{Proof of Theorem \ref{T5.2}} \label{App-C}
\setcounter{equation}{0}
By Proposition \ref{P5.4}, we may assume that there exists an entrance
time $t_{0}$ such that 
\[
\mathcal{D}(\Theta(t))\le\frac{\pi}{2}-\sin^{-1}\delta,\quad t\ge t_{0},
\]
for some $0<\delta<1$. For the derivation of desired exponential
decay, we split its proof into four steps. \\

$\bullet$ Step A (A differential inequality for some $\omega_{i}$):~We
set 
\[
\overline{\omega}(t):=\sup_{n\in\mathbb{N}}\omega_{n}(t).
\]
Let $i\in\mathbb{N}$ be an index such that 
\begin{equation}
	\omega_{i}(t_{0})>\frac{3}{4}\overline{\omega}(t_{0}).\label{E-21-0-0}
\end{equation}
Since 
\begin{equation}
	\sum_{j\in\mathbb{N}}\kappa_{j}\omega_{j}(t_{0})=0\quad\mbox{with}~\left\{ \omega_{i}\right\} _{i\in\mathbb{N}}\not\equiv\boldsymbol{0},\label{E-21-0}
\end{equation}
such $i$ exists. We set 
\begin{equation}
	J(t_{0}):=\left\{ j\in\mathbb{N}:~\omega_{j}(t_{0})\ge\omega_{i}(t_{0})\right\} .\label{E-21-1}
\end{equation}
Note that 
\begin{align}
	\begin{aligned}\dot{\omega}_{i}(t_{0}) & =\sum_{j\in\mathbb{N}}\kappa_{j}\cos\left(\theta_{i}(t_{0})-\theta_{j}(t_{0})\right)\left(\omega_{j}(t_{0})-\omega_{i}(t_{0})\right)\\
		& =\sum_{j\in J(t_{0})}\kappa_{j}\cos\left(\theta_{i}(t_{0})-\theta_{j}(t_{0})\right)\left(\omega_{j}(t_{0})-\omega_{i}(t_{0})\right)\\
		& \hspace{0.5cm}+\sum_{j\in\mathbb{N}\setminus J(t_{0})}\kappa_{j}\cos\left(\theta_{i}(t_{0})-\theta_{j}(t_{0})\right)\left(\omega_{j}(t_{0})-\omega_{i}(t_{0})\right)\\
		& :=\mathcal{I}_{21}+\mathcal{I}_{22}.
	\end{aligned}
	\label{E-22}
\end{align}
Below, we estimate the term $\mathcal{I}_{2i}$ with $i=1,2$. \\

$\diamond$~(Estimate of $\mathcal{I}_{21}$): We use \eqref{E-21-0}
and \eqref{E-21-1} to get 
\begin{align}
	\begin{aligned}\mathcal{I}_{21} & =\sum_{j\in J(t_{0})}\kappa_{j}\cos\left(\theta_{i}(t_{0})-\theta_{j}(t_{0})\right)\left(\omega_{j}(t_{0})-\omega_{i}(t_{0})\right)\\
		& \le\sum_{j\in J(t_{0})}\kappa_{j}\left(\omega_{j}(t_{0})-\omega_{i}(t_{0})\right)=-\sum_{j\in\mathbb{N}\setminus J(t_{0})}\kappa_{j}\omega_{j}(t_{0})-\sum_{j\in J(t_{0})}\kappa_{j}\omega_{i}(t_{0}).
	\end{aligned}
	\label{E-23}
\end{align}

\vspace{0.2cm}

$\diamond$~(Estimate of $\mathcal{I}_{22}$): Again we use \eqref{E-21-0}
to obtain 
\begin{align}
	\begin{aligned}\mathcal{I}_{22} & =\sum_{j\in\mathbb{N}\setminus J(t_{0})}\kappa_{j}\cos\left(\theta_{i}(t_{0})-\theta_{j}(t_{0})\right)\left(\omega_{j}(t_{0})-\omega_{i}(t_{0})\right)\\
		& \le\delta\sum_{j\in\mathbb{N}\setminus J(t_{0})}\kappa_{j}\left(\omega_{j}(t_{0})-\omega_{i}(t_{0})\right)=-\delta\sum_{j\in J(t_{0})}\kappa_{j}\omega_{j}(t_{0})-\delta\sum_{j\in\mathbb{N}\setminus J(t_{0})}\kappa_{j}\omega_{i}(t_{0}),
	\end{aligned}
	\label{E-24}
\end{align}
the equality in \eqref{E-24} holds by \eqref{E-21-0}. In \eqref{E-22},
we combine all the estimates \eqref{E-23} and \eqref{E-24} to get
\begin{align*}
	\dot{\omega}_{i}(t_{0}) & \le-\sum_{j\in J(t_{0})}\kappa_{j}\omega_{i}(t_{0})-\sum_{j\in\mathbb{N}\setminus J(t_{0})}\kappa_{j}\omega_{j}(t_{0})-\delta\sum_{j\in J(t_{0})}\kappa_{j}\omega_{j}(t_{0})-\delta\sum_{j\in\mathbb{N}\setminus J(t_{0})}\kappa_{j}\omega_{i}(t_{0})\\
	& =-\sum_{j\in J(t_{0})}\kappa_{j}\omega_{i}(t_{0})-\delta\sum_{j\in\mathbb{N}\setminus J(t_{0})}\kappa_{j}\omega_{i}(t_{0})-(1-\delta)\sum_{j\in\mathbb{N}\setminus J(t_{0})}\kappa_{j}\omega_{j}(t_{0})\\
	& =-\delta\|K\|_{\infty,1}\omega_{i}(t_{0})+(1-\delta)\sum_{j\in J(t_{0})}\kappa_{j}\left(\omega_{j}(t_{0})-\omega_{i}(t_{0})\right)<-\delta\|K\|_{\infty,1}\omega_{i}(t_{0}).
\end{align*}

\vspace{0.5cm}

$\bullet$ Step B~ (Estimate $\omega_{i}(t)$ for $t\ll1$): For
$i\in\mathbb{N}$ satisfying the relation \eqref{E-21-0-0}, we set
\begin{equation}
	C_{1}:=\frac{\delta}{16\left(\left\Vert \mathcal{V}\right\Vert _{\infty}+2\|K\|_{\infty,1}\right)}\min\left\{ 1,\frac{1}{\left\Vert K\right\Vert _{\infty,1}}\right\} .\label{E-25}
\end{equation}
In the sequel, we estimate $\omega_{i}$ in the time interval $[t_{0},t_{0}+C_{1}\overline{\omega}(t_{0})]$.
\\

First, we use Lemma \ref{L2.2} and \eqref{E-21-0-0} to find 
\begin{align*}
	\begin{aligned}\dot{\omega}_{i}(t) & \le\dot{\omega}_{i}(t_{0})+2\|K\|_{\infty,1}\left(\left\Vert \mathcal{V}\right\Vert _{\infty}+\|K\|_{\infty,1}\right)(t-t_{0})\\
		& \le-\frac{3\delta}{4}\|K\|_{\infty,1}\overline{\omega}(t_{0})+2\|K\|_{\infty,1}\left(\left\Vert \mathcal{V}\right\Vert _{\infty}+\|K\|_{\infty,1}\right)(t-t_{0})\\
		& \le-\frac{3\delta}{8}\|K\|_{\infty,1}\overline{\omega}(t_{0}),
	\end{aligned}
\end{align*}
where we used \eqref{E-25}. This yields 
\begin{equation}
	\omega_{i}(t)\le\overline{\omega}(t_{0})-\frac{3\delta}{8}\|K\|_{\infty,1}\overline{\omega}(t_{0})\left(t-t_{0}\right).\label{E-26}
\end{equation}
Next, we consider an index $i\in\mathbb{N}$ such that 
\[
\omega_{i}(t_{0})\le\frac{3}{4}\overline{\omega}(t_{0}).
\]
In this case, we use \eqref{E-25} to get 
\begin{align*}
	\omega_{i}(t) & \le\omega_{i}(t_{0})+2\|K\|_{\infty,1}\left(\left\Vert \mathcal{V}\right\Vert _{\infty}+2\|K\|_{\infty,1}\right)\left(t-t_{0}\right)\\
	& \le\frac{3}{4}\overline{\omega}(t_{0})+\frac{1}{8}\overline{\omega}(t_{0}).
\end{align*}
Hence we obtain
\begin{equation}
	\overline{\omega}(t)\le\max\left\{ \frac{7}{8},\left(1-\frac{3\delta}{8}\|K\|_{\infty,1}(t-t_{0})\right)\right\} \overline{\omega}(t_{0})\label{E-27}
\end{equation}
for $t_{0}\le t<t_{0}+C_{1}\overline{\omega}(t_{0})$.\\

$\bullet$ Step C~($\overline{\omega}(t)$ is nonincreasing for $t\ge t_{0}$):~Let
$t_{1}>t_{0}$ and 
\[
I=\left\{ t:\overline{\omega}(t)\le\overline{\omega}(t_{1})\right\} 
\]
in $[t_{1},\infty)$. By Lemma \ref{L2.2}, $\ddot{\Theta}=\dot{\mathcal{W}}$ is uniformly bounded and $\omega_{i}(t)$ is globally Lipschitz. Applying similar process with Lemma \ref{L2.2} gives $\overline{\omega}$ is also Lipschitz.
Hence $I$ is nonempty closed subset of $[t_{1},\infty)$, and Step
B proves that $I$ is open in $[t_{1},\infty)$. Hence $\overline{\omega}(t)$
is globally decreasing.\\

$\bullet$~Final step~($\overline{\omega}$ is exponentially decreasing for $t\ge t_{0}$.):
Let 
\[
t_{n}=t_{n-1}+C_{1}\overline{\omega}(t_{n-1}),\quad n\in\mathbb{N}.
\]
Then,
\begin{align*}
	& \left(1-\frac{3\delta}{8}\|K\|_{\infty,1}C_{1}\overline{\omega}(t_{n-1})\right)<\frac{7}{8}\iff\frac{1}{3\delta\|K\|_{\infty,1}C_{1}}<\overline{\omega}(t_{n-1})
\end{align*}
gives
\[
\overline{\omega}(t_{n})<\frac{7}{8}\overline{\omega}(t_{n-1})\iff\frac{1}{3\delta\|K\|_{\infty,1}C_{1}}<\overline{\omega}(t_{n-1})
\]
in \eqref{E-27}. Combined with
\[
t_{n+1}-t_{n}<t_{n}-t_{n-1}\iff\overline{\omega}(t_{n})<\overline{\omega}(t_{n-1}),
\]
we can conclude that 
\[
\overline{\omega}(t+n(t_{1}-t_{0}))<\overline{\omega}(t_{n})<\left(\frac{7}{8}\right)^{n}\overline{\omega}(t_{0})
\]
for
\[
n\in\left\{ m\in\mathbb{N}:\frac{1}{3\delta\|K\|_{\infty,1}C_{1}}<\overline{\omega}(t_{m})\right\} .
\]
Now it is enough to show the conclusion with assuming $\frac{1}{3\delta\|K\|_{\infty,1}C_{1}}\ge\overline{\omega}(t_{0})$.

We choose $k\in\mathbb{N}$ such that 
\[
\frac{1}{2^{k-1}}\ge\overline{\omega}\left(t_{0}\right)>\frac{1}{2^{k}}.
\]
Suppose that 
\[
\overline{\omega}\left(t\right)>\frac{1}{2^{k}}\quad\mbox{for \ensuremath{t>t_{0}}}.
\]
By induction on $n$, we can show that 
\[
\overline{\omega}(t)\le\overline{\omega}(\tilde{t}_{0,n})-\frac{3}{32}\|K\|_{\infty,1}\cdot\frac{1}{2^{k}}\cdot\left(t-\tilde{t}_{0,n}\right),\quad\tilde{t}_{0,n}\le t\le\tilde{t}_{0,n+1},
\]
where $\tilde{t}_{0,n}:=t_{0}+n\cdot\frac{C_{1}}{2^{k}}$. \\

This implies 
\[
\overline{\omega}(t)\le\overline{\omega}(t_{0})-\frac{3}{32}\|K\|_{\infty,1}\cdot\frac{1}{2^{k}}\cdot n,\quad\forall~n\in\mathbb{N},
\]
which is contradictory to $\overline{\omega}\left(t\right)>\frac{1}{2^{k}}$.
\\

Furthermore, we have 
\[
\overline{\omega}(t)\le\overline{\omega}(t_{0})-\frac{3}{32}\|K\|_{\infty,1}\cdot\frac{1}{2^{k}}\cdot n\le\frac{1}{2^{k-1}}-\frac{3}{32}\|K\|_{\infty,1}\cdot\frac{1}{2^{k}}\cdot n\le\frac{1}{2^{k}},
\]
for { $n\ge\lfloor\frac{32}{3\|K\|_{\infty,1}}\rfloor+1.$} \\

This implies 
\[
\inf\left\{ t:\overline{\omega}(t)\le\frac{1}{2^{k}}\right\} \le t_{0}+\Big\lfloor\frac{32}{3\|K\|_{\infty,1}}\Big\rfloor+1.
\]
Inductively, we can derive 
\begin{equation}
	0<\overline{\omega}(t)\le\frac{1}{2^{n}}\overline{\omega}(t_{0}),\qquad t\ge t_{0}+n\left(\Big\lfloor\frac{32}{3\|K\|_{\infty,1}}\Big\rfloor+1\right).\label{E-29}
\end{equation}
Similarly, we set 
\[
\underline{\omega}(t_{0}):=\inf_{n\in\mathbb{N}}\omega_{n}(t_{0}).
\]
This yields 
\begin{equation}
	0>\underline{\omega}(t)\ge\frac{1}{2^{n}}\underline{\omega}(t_{0}),\quad\mbox{for}~t\ge t_{0}+n\left(\Big\lfloor\frac{32}{3\|K\|_{\infty,1}}\Big\rfloor+1\right).\label{E-30}
\end{equation}
Finally, we combine \eqref{E-29} and \eqref{E-30} to find 
\[
\mathcal{D}\left(\mathcal{W}(t)\right)=\overline{\omega}(t)-\underline{\omega}(t)\le\frac{1}{2^{n}}\left(\overline{\omega}(t_{0})-\underline{\omega}(t_{0})\right)=\frac{1}{2^{n}}\mathcal{D}\left(\mathcal{W}(t_{0})\right),
\]
for~$t\ge t_{0}+n\left(\lfloor\frac{32}{3\|K\|_{\infty,1}}\rfloor+1\right).$
In this case, we have exponential synchronization: 
\[
\mathcal{D}\left(\mathcal{W}(t)\right)\le\mathcal{D}\left(\mathcal{W}(t_{0})\right)\cdot\exp\left[-\frac{3\|K\|_{\infty,1}\log2}{32}\left(t-t_{0}\right)+1\right],\quad t\ge t_{0}.
\]

\end{appendices}

%%===========================================================================================%%
%% If you are submitting to one of the Nature Portfolio journals, using the eJP submission   %%
%% system, please include the references within the manuscript file itself. You may do this  %%
%% by copying the reference list from your .bbl file, paste it into the main manuscript .tex %%
%% file, and delete the associated \verb+\bibliography+ commands.                            %%
%%===========================================================================================%%

%\bibliography{sn-bibliography}% common bib file
%% if required, the content of .bbl file can be included here once bbl is generated
%%\input sn-article.bbl

\end{document}